\documentclass[table]{amsart} 
\usepackage{amsmath}
\usepackage[margin=1.0in]{geometry}
\usepackage{amscd,amsmath,amsxtra,amsthm,amssymb,stmaryrd,xr,mathrsfs,mathtools,enumerate,commath, comment, enumitem, graphicx}
\usepackage{amsfonts, amssymb, amsthm, amsmath, braket, hyperref, mathtools, comment, mathrsfs, enumitem, cleveref}
\usepackage{stmaryrd}
\usepackage{amsfonts}
\usepackage{orcidlink}
\usepackage{multirow}
\usepackage{xcolor}
\usepackage{commath}
\usepackage{comment}
\usepackage{tikz-cd}
\usepackage{tkz-graph}
\usepackage{colonequals}
\usepackage{hyperref}
\usepackage{graphicx}
\usepackage{bigstrut}
\usepackage{tabularx}
\numberwithin{equation}{section}
\usepackage{longtable, multirow, array}
\usepackage{cleveref}
\crefname{section}{§}{§§}
\Crefname{section}{§}{§§}
\newcommand{\genlegendre}[4]{%
	\genfrac{(}{)}{}{#1}{#3}{#4}%
	\if\relax\detokenize{#2}\relax\else_{\!#2}\fi
}
\newcommand{\legendre}[3][]{\genlegendre{}{#1}{#2}{#3}}

\usepackage{longtable} 
\usepackage{pdflscape} 
\usepackage{booktabs}
\usepackage{hyperref}
\definecolor{vegasgold}{rgb}{0.77, 0.7, 0.35}
\definecolor{darkgoldenrod}{rgb}{0.72, 0.53, 0.04}
\definecolor{gold(metallic)}{rgb}{0.83, 0.69, 0.22}
\hypersetup{
 colorlinks=true,
 linkcolor=darkgoldenrod,
 filecolor=brown,      
 urlcolor=gold(metallic),
 citecolor=darkgoldenrod,
 }

\usepackage[all,cmtip]{xy}
\DeclareFontFamily{U}{wncy}{}
\DeclareFontShape{U}{wncy}{m}{n}{<->wncyr10}{}
\DeclareSymbolFont{mcy}{U}{wncy}{m}{n}
\DeclareMathSymbol{\Sh}{\mathord}{mcy}{"58}
\usepackage[T2A,T1]{fontenc}
\usepackage[OT2,T1]{fontenc}

\usepackage{tikz}
\usetikzlibrary{shapes.geometric}
\usetikzlibrary{decorations.markings}
\tikzset{every loop/.style={min distance=10mm,looseness=10}}
\tikzstyle{vertex}=[auto=left,circle,minimum size=1pt,inner sep=0pt]

\newtheorem{theorem}{Theorem}[section]
\newtheorem{lemma}[theorem]{Lemma}

\newtheorem*{theorem*}{Theorem}
\newtheorem*{ass*}{Assumption}

\newtheorem{corollary}[theorem]{Corollary}
\newtheorem{remark}[theorem]{Remark}

\newtheorem{proposition}[theorem]{Proposition}

\newcommand{\Gal}{\mathrm{Gal}}

\newcommand{\Z}{\mathbb{Z}}

\newcommand{\Q}{\mathbb{Q}}
\newcommand{\F}{\mathbb{F}}

\newcommand{\Sel}{\mathrm{Sel}}

\newcommand{\rank}{\mathrm{rank}}

\DeclareFontEncoding{OT2}{}{}
\DeclareSymbolFont{cyrletters}{OT2}{wncyr}{m}{n}
\DeclareMathSymbol{\Sha}{\mathalpha}{cyrletters}{"58}

\numberwithin{equation}{section}

\begin{document}

\title{$2$-Selmer groups, $2$-class groups, and congruent numbers}
\author[Shamik Das]{Shamik Das}
\address[Das]{Department of Mathematics and Statistics, IIT Kanpur, India}
\email{shamikd@iitk.ac.in}

\author[Debajyoti De]{Debajyoti De}
\address[De]{Department of Mathematics, IIT Madras, India}
\email{debajyotide20@gmail.com}

\author[Sudipa Mondal]{Sudipa Mondal}
\address[Mondal]{Department of Mathematics, IIT Madras, India}
\email{sudipa.mondal123@gmail.com}


\begin{abstract}
In this article, we study necessary conditions for certain square-free integers to be congruent numbers. Our method uses divisibility properties of class numbers of related imaginary quadratic fields. We first consider positive square-free integers of the form $n = p_1 p_2 \cdots p_t q,$ where each prime $p_i \equiv 5 \pmod{8}$ and $q \equiv 7 \pmod{8}$. We show that if such an integer $n$ is a congruent number, then the class number $h(-n)$ of the quadratic field $\mathbb{Q}(\sqrt{-n})$ satisfies a specific divisibility condition. Furthermore, we provide quantitative lower bounds on the number of non-congruent numbers of this form.  Next, we study integers of the form  $n = p_1 p_2 \cdots p_t q,$ with $p_i \equiv 5 \pmod{8}$ and $q \equiv 3 \pmod{8}$. Assuming that $n$ is a congruent number, we obtain a congruence modulo powers of $2$ between the class numbers of the fields $\mathbb{Q}(\sqrt{-n})$ and $\mathbb{Q}\!\left(\sqrt{-p_1 p_2 \cdots p_t}\right)$.
\end{abstract}

\maketitle

\section{Introduction}
One of the oldest problems in the history of number theory is to determine which positive integers $n$ can occur as the common difference of a three-term arithmetic progression of rational squares. That is, given a positive integer $n$, find a rational number $a \in \mathbb{Q}^\times$ such that 
both $a^2 \pm n$ are rational squares. Such positive integers $n$ are called \emph{congruent numbers}. There is a one-to-one correspondence between right triangles with area $n$ and three-term arithmetic progressions of squares with common difference $n$. More precisely, the sets
$$\left \{(a,b,c): a^2 + b^2 = c^2,\ \frac{ab}{2} = n \right\} \quad \text{and} \quad
\{(r,s,t): s^2 - r^2 = n,\ t^2 - s^2 = n\}$$
are in one-to-one correspondence by
$$(a,b,c) \mapsto \left(\frac{b-a}{2},\, \frac{c}{2},\, \frac{b+a}{2}\right),\qquad (r,s,t) \mapsto (t-r,\, t+r,\, 2s).$$ 
Hence, a positive integer is a congruent number if it is the area of a right-angled triangle whose sides all have rational length. The problem of determining whether a given positive integer is a congruent number or not is known as the \emph{congruent number problem}. At present, no algorithm is known for determining whether a given positive integer $n$ is congruent. It is clear that $n$ is a congruent number if and only if its square-free part is a congruent number. Hence, we restrict our attention to square-free integers. 
It turns out that the congruent number property is also equivalent to the existence of a nontrivial rational solution of the equation $y^2 = x^3-n^2 x.$ This equation has three obvious rational solutions, namely $(0,0), (n,0),$ and $(-n,0)$, all lying on the line $y=0$. For $n >0$, there is a one-to-one correspondence between the following two sets:
$$\left\{(a,b,c) : a^2 + b^2 = c^2,\ \frac{ab}{2} = n \right \} \quad \text{and} \quad \{(x,y) : y^2 = x^3 - n^2 x,\ y \neq 0\}.$$
Mutually inverse correspondences between these sets are
$$(a,b,c) \mapsto \left(\frac{n b}{\,c - a\,},\ \frac{2n^2}{\,c - a\,}\right),\qquad (x,y) \mapsto \left(\frac{x^2 - n^2}{y},\ \frac{2nx}{y},\ \frac{x^2 + n^2}{y}\right).$$
Let $(x,y)$ be a torsion point on the elliptic curve \begin{equation}\label{Congruent curve}
    E_n : y^{2} = x^{3} - n^{2}x.
\end{equation} 
By the Nagell–Lutz theorem \cite[Corollary~7.2]{silverman2009arithmetic}, we have $x,y \in \mathbb{Z}$ and either $y=0$ or $y^{2}\mid 4n^{6}$.  
The Diophantine equation
\[
    y^{2} = x^{3} - n^{2}x, \quad x,y \in \mathbb{Z},\ y\neq 0, \ y^{2}\mid 4n^{6},
\]
admits no solutions. Hence the torsion subgroup of $E_n(\mathbb{Q})$ is precisely $E_n(\mathbb{Q})_{\mathrm{tors}} = \{\mathcal{O},\ (0,0),\ (n,0),\ (-n,0)\}.$
Therefore every rational point $(x,y)$  with $y\neq 0$ on $E_n$ is of infinite order.  
Consequently, a positive square-free integer $n$ is a congruent number if and only if the elliptic curve $E_{n}$ has a rational point of infinite order. The elliptic curve \eqref{Congruent curve} is referred to as the congruent number elliptic curve. By the celebrated Mordell--Weil theorem, the group $E_n(\mathbb{Q})$ is a finitely generated abelian group and therefore $E_n(\mathbb{Q}) \cong \mathbb{Z}^{\,r_n} \oplus E_n(\mathbb{Q})_{\mathrm{tors}}$. The non-negative integer $r_n$ is called the \emph{algebraic rank} of $E_n(\Q)$. 
Consequently, a positive square-free integer $n$ is a congruent number if and only if $r_n>0$.


It is well-known that the order of vanishing $R_n$ of the Hasse-Weil $L$-function $L(E_n,s)$ at $s=1$ is referred to as the \emph{analytic rank} of $E_n$. Moreover, the $L$-function $L(E_n,s)$ satisfies a functional equation that relates 
its values at $s$ and $2-s$ through a sign $\omega(E_n)=\pm 1$, known as the \emph{root number} of $E_n$. According to a weak form of the Birch–Swinnerton–Dyer conjecture, the Shafarevich–Tate group $\Sha(E_n/\mathbb{Q})$ is finite, and the analytic rank $R_n$ equals the algebraic rank $r_n$ for the elliptic curve $E_n$.
This deep conjecture connects the arithmetic of $E_n$ with the analytic properties 
of its $L$-function.  
The parity conjecture further asserts that
$(-1)^{r_n} = \omega(E_n).$ In particular, if $\omega(E_n)=-1$, then $r_n$ is positive, and hence 
$E_n(\mathbb{Q})$ contains infinitely many rational points.
 Birch and Stephens \cite{Birch-Stephens} showed that 
$$\omega(E_n)=
  \begin{cases}
     +1  &  \text{if}~~ n \equiv 1,2,3 \pmod 8, \\
     -1  &  \text{if}~~ n \equiv 5,6,7 \pmod 8.
 \end{cases} $$
If one assumes the Birch--Swinnerton-Dyer conjecture, any integer $n \equiv 5,6,7 \pmod{8}$ must be a congruent number. This reduces the problem to the cases $n \equiv 1,2,3 \pmod{8}$, which are therefore of greater interest. Tunnell \cite{tunnell1983classical} gave a characterization of congruent numbers, assuming the BSD conjecture.

Let $p$ be a prime.  
Heegner \cite{heegner1952diophantische} and Birch \cite{birch1969diophantine} proved that 
$2p$ is a congruent number whenever $p \equiv 3 \pmod{4}$.  
On the other hand, Genocchi \cite{genocchi1855note} showed that primes 
$p \equiv 3 \pmod{8}$ themselves are non-congruent.  
Using the theory of mock Heegner points, Monsky \cite{monsky1990mock} prove that 
primes $p \equiv 5,7 \pmod{8}$ are congruent numbers. Lagrange \cite{lagrange1975nombres} provided a criterion for determining non-congruent 
numbers of the form $pq$: if $\legendre{p}{q}=-1$, where 
$p \equiv 5 \pmod{8}$ and $q \equiv 7 \pmod{8}$ are primes, then the integer $pq$ is 
non-congruent.  
Qin \cite{MR4458410} and Das--Saikia \cite{das2025class} studied integers of the form $pq$ 
and proved, by different methods, that if $pq$ is a congruent number, then the class number 
of the imaginary quadratic field $\mathbb{Q}(\sqrt{-pq})$ is divisible by $8$.  

Furthermore, Lagrange \cite{lagrange1975nombres} also established that if $p,q,r$ are 
distinct primes satisfying $p \equiv q \equiv 5 \pmod{8}$ and $r \equiv 3 \pmod{8}$, 
then $pqr$ is non-congruent whenever any one of the following conditions holds:
\[
\legendre{p}{q} = \legendre{p}{r} = -1, \qquad
\legendre{q}{p} = \legendre{q}{r} = -1, \qquad
\legendre{r}{p} = \legendre{r}{q} = -1.
\]
The case of square-free integers $n \equiv 3 \pmod{8}$ of the form $n = p_1 p_2 \cdots p_t q$, where the primes $p_i\equiv 1 \pmod 8$ and $q \equiv 3 \pmod{8}$ are distinct, has been recently considered by Das and Mondal \cite{DasMondalInPrep}. In the present article, we consider positive square-free integers of the form $n = p_1 p_2 \dotsm p_t q,$
where either each prime $p_i \equiv 5 \pmod{8}$ and $q \equiv 7 \pmod{8}$ (see Theorem \ref{5557}), or each 
$p_i \equiv 5 \pmod{8}$ and $q \equiv 3 \pmod{8}$ with the additional condition that $n \equiv 3 \pmod{8}$ (see Theorem \ref{553}).  
Before stating the main results, we introduce some notations.

\begin{itemize}
    \item $C(-n)$: the ideal class group of the imaginary quadratic field $\mathbb{Q}(\sqrt{-n})$.
    
    \item $h(-n)$: the ideal class number of $\mathbb{Q}(\sqrt{-n})$.

    \item $D(-n)$: the discriminant of the quadratic field $\mathbb{Q}(\sqrt{-n})$.

    \item $G[m]$: the $m$-torsion subgroup of an abelian group $G$.

    \item $\legendre{\cdot}{\cdot}$ and $\legendre[4]{\cdot}{\cdot}$: the Legendre symbol and the quartic residue symbol, respectively.


    \item $M_k(\mathbb{F})$ and $M_{k \times \ell}(\mathbb{F})$: the sets of all $k \times k$ and $k \times \ell$ matrices over a field $\mathbb{F}$, respectively.

    \item Let $I_t$ denote the $t \times t$ identity matrix and $O_t$ denote the $t \times t$ zero matrix.


    \item $\F_2:$ denotes the finite field with two elements.

    \item For distinct primes $p_1, p_2, \ldots, p_t$, set $P := p_1 p_2 \cdots p_t$.  
    Define the $t \times t$ matrix $A_P = [a_{ij}]$ by
    \begin{equation}\label{matrix A_n 5557}
        a_{ij} =
        \begin{cases}
            1, & \text{if } \legendre{p_j}{p_i} = -1 \text{ and } i \neq j,\\[4pt]
            0, & \text{if } \legendre{p_j}{p_i} = 1 \text{ and } i \neq j,
        \end{cases}
        \qquad\text{and}\qquad
        a_{ii} = \sum_{j \ne i} a_{ij}.
    \end{equation}
    In the case $t = 1$, we set $A_P$ to be the $1 \times 1$ zero matrix, $A_P = [0]$.
   \item For $r\ge 1$, let $\mathbf{1}_r := (1,\ldots,1)^{\mathsf T} \in \mathbb{F}_2^{\,r}$ and
$\mathbf{0}_r := (0,\ldots,0)^{\mathsf T} \in \mathbb{F}_2^{\,r}$.
\end{itemize}



Our first main result is as follows:
\begin{theorem}\label{5557}
    Let $p_{1},\dotsc,p_{t}$ and $q$ be distinct primes such that
$p_{i} \equiv 5 \pmod{8}$ for all $1 \le i \le t$, $q \equiv 7 \pmod{8}$,
and $t$ is odd. Set $n = p_{1}p_{2}\cdots p_{t} q=Pq$. Assume further that:
\begin{enumerate}
    \item[(a)] $\rank_{\mathbb{F}_{2}}(A_{P}) = t - 1$, where $A_{P}$ is a $t \times t$ matrix defined as in \eqref{matrix A_n 5557};
    \item[(b)] $\displaystyle \legendre{p_i}{q} = 1$ for all $i \in \{1,2,\dotsc,t\}$.
\end{enumerate}
If $n$ is a congruent number, then
\begin{equation*}
\label{5557-equation}
  h(-n) \equiv 0 \pmod{2^{\,t+2}}.  
\end{equation*}
\end{theorem}

\begin{theorem}\label{553}
Let $p_{1},\dotsc,p_{t}$ and $q$ be distinct primes such that
$p_{i} \equiv 5 \pmod{8}$ for all $1 \le i \le t$, $q \equiv 3 \pmod{8}$,
and $t$ is even. Set $P= p_{1}p_{2}\cdots p_{t}$ and $n = Pq$. Assume further that:
\begin{enumerate}
    \item[(a)] $\rank_{\mathbb{F}_{2}}(A_{P}) = t - 1$, where $A_{P}$ is a $t \times t$ matrix defined as in  \eqref{matrix A_n 5557};
    \item[(b)] $\displaystyle \legendre{p_i}{q} = 1$ for all $i \in \{1,2,\dotsc,t\}$;
    \item[(c)] the $2$-torsion subgroup $\Sha(E_n/\mathbb{Q})[2]$ of the Shafarevich--Tate group of $E_n$ is finite.
\end{enumerate}
If $n$ is a congruent number, then
\[ h(-n)  \equiv h(-P) + 2^{t+1} \pmod{2^{\,t+2}}. \]
\end{theorem}

\begin{remark}
{\rm
Observe that, in contrast to Theorem~\ref{5557}, Theorem~\ref{553} requires an additional hypothesis on the finiteness of the $2$-torsion subgroup 
$\Sha(E_n/\Q)[2]$ of the Shafarevich--Tate group of $E_n$. This assumption is expected to hold for elliptic curves $E_n$ for all square-free integer $n$ in view of the Birch--Swinnerton-Dyer conjecture.

The converses of our results, however, are not true. In particular, the criteria we obtain are necessary for $n$ to be a congruent number, but they are not sufficient.

For example, consider  $n = 42090427 = 53 \cdot 61 \cdot 277 \cdot 47,$
which satisfies all the hypotheses of Theorem~\ref{5557} with $t = 3$. One verifies that  
$h(-n) = 704 \equiv 0 \pmod{2^{5}},$
as predicted, although $n$ is known to be a non-congruent number.

Similarly, take  $n = 70115 = 5 \cdot 37 \cdot 379, \; P = 5 \cdot 37,$ which satisfies all the assumptions of Theorem~\ref{553} with $t = 2$. In this case, $h(-n) = 88, \; h(-P) = 16,$
and indeed, $h(-n) \equiv h(-P)+ 8 \pmod{2^{4}},$
yet $n$ is a non-congruent number.
}
\end{remark}

\subsection{Organization}
Both Theorem~\ref{5557} and Theorem~\ref{553} are proved using a congruence relation between class numbers (see Lemmas \ref{8 rank condition}, \ref{8-rank of 553}, and \ref{8 rank of n_q}). The main idea of the proof is as follows. We study certain systems of Diophantine equations arising from the image of the Mordell--Weil group $E_n(\mathbb{Q})$ under the $2$-descent map, as explained in Section~\ref{Selmer group and Monsky}. In the same section, we also describe how to compute the $2$-Selmer rank using Monsky matrices.

In Section~\ref{4 rank and 8 rank}, we review basic facts about the $4$-rank and $8$-rank of the ideal class group of imaginary quadratic fields. The proofs of our main results are given in Section~\ref{proof of the theorem}. In Section~\ref{Quantitative estimate}, we give lower bounds for the number of non-congruent numbers satisfying the conditions of Theorem~\ref{5557}. Finally, in Section~\ref{Computation}, we present examples that support Theorems~\ref{5557} and~\ref{553} along with possible future directions for further research. 


\section{Selmer Group and Monsky Matrix}\label{Selmer group and Monsky}
Suppose $n$ is a square-free positive integer. In this section, we study the Mordell-Weil rank and the $2$-Selmer rank of the congruent number elliptic curve $E_{n}$ for a square-free integer $n$. We also discuss how the $2$-Selmer rank is connected to the rank of the Monsky matrix associated to  $n$, and investigate the solvability (local and global) of a system of equations that arise in $2$-descent of $E_{n}$.


Let $E_n/\Q$ be the congruent number elliptic curve defined by $E_n:\ y^2 = x^3 - n^2 x .$ For a place $v$ of $\Q$, let $\Q_v$ denote the corresponding completion. 
We write $\overline{\Q}$ and $\overline{\Q}_v$ for the algebraic closures of $\Q$ and $\Q_v$, respectively. 
Let $G_\Q := \Gal(\overline{\Q}/\Q)$ and  $G_{\Q_v} := \Gal(\overline{\Q}_v/\Q_v)$
denote the absolute Galois group of $\Q$ and the decomposition subgroup at the place $v$, respectively. The Tate--Shafarevich group of $E_n$ over $\Q$ is defined by
\[
\Sha(E_n/\Q)
:=
\ker\!\left(
H^1(G_\Q, E_n(\overline{\Q}))
\longrightarrow
\prod_{v} H^1\!\left(G_{\Q_v}, E_n(\overline{\Q}_v)\right)
\right),
\]
where the product is taken over all places $v$ of $\Q$.
The $2$-Selmer group of $E_n$ over $\Q$ is defined as
\[
\Sel_2(E_n/\Q)
:=
\ker\!\left(
H^1\!\left(G_\Q, E_n(\overline{\Q})[2]\right)
\longrightarrow
\prod_{v} H^1\!\left(G_{\Q_v}, E_n(\overline{\Q}_v)\right)
\right).
\]
For further details, we refer the reader to \cite[Chapter~X]{silverman2009arithmetic}. 
Let $r_n$ denote the algebraic rank of $E_n(\Q)$. 
The $2$-descent on $E_n$ yields the exact sequence
\begin{equation}\label{Descent-exact-seq}
0 \longrightarrow E_n(\Q)/2E_n(\Q)
\longrightarrow \Sel_2(E_n/\Q)
\longrightarrow \Sha(E_n/\Q)[2]
\longrightarrow 0 .
\end{equation}
Since
$E_n(\Q)[2] \cong (\Z/2\Z) \oplus (\Z/2\Z),$
it follows that
$\#\bigl(E_n(\Q)/2E_n(\Q)\bigr) = 2^{r_n+2}.$ Therefore, the $2$-Selmer group $\Sel_2(E_n/\Q)$ is a finite $2$-group. We write
$\#\Sel_2(E_n/\Q)=2^{2+s_n},$
where $s_n$ denotes the $2$-Selmer rank of $E_n$. Taking $\F_2$-dimensions in the exact sequence
\eqref{Descent-exact-seq}, we obtain
\begin{equation}\label{relation-of-selmer-sha}
s_n = r_n + \dim_{\F_2}\Sha(E_n/\Q)[2].
\end{equation}
Since $\dim_{\F_2}\Sha(E_n/\Q)[2]\ge 0$, this immediately implies
\begin{equation}\label{inequality-selmer}
0 \le r_n \le s_n .
\end{equation}

 Now if $\Sha(E_n/\Q)[2]=0$, then by Cassel-Tate pairing \cite{cassels1962_arithmetic_iv}, we have $\dim_{\F_2}\Sha(E_n/\Q)[2]$ is even. Hence \begin{equation}
\label{relation of selmer sha}
 s_n \equiv r_n \pmod 2.   
 \end{equation}
  So the $2$-Selmer rank $s_n$ and the Mordell–Weil rank $r_n$ have the same parity.


Consider the group isomorphism
\begin{equation}\label{eq:phi-def}
\phi:\{\pm 1\} \longrightarrow \{0,1\}, \quad \text{ defined by } \phi\bigl((-1)^{\epsilon}\bigr)=\epsilon,
\quad \text{for } \epsilon \in \{0,1\}.
\end{equation} 
In this context, $\{\pm 1\}$ is regarded as a multiplicative group, while $\{0,1\}$ is viewed as an additive group modulo $2$. 

Let $r$ be a positive integer, and $n = p_1 p_2 \dotsb p_r$ be a square-free positive integer with distinct odd prime factors $p_1, p_2, \dotsc, p_r$. 
For $l \in \{\pm 2\}$, we define an $r \times r$ diagonal matrix $D_l = [d_i]$, and an $r \times r$ matrix $A = [a_{ij}]$, respectively, as follows:

\begin{equation}
\label{diaij}
    d_i= \phi\left(\legendre{l}{p_i}\right), \quad \quad a_{ij}= \begin{cases}
    \phi\left(\legendre{p_j}{p_i}\right) & \text{if}~ i \neq j,\\
     \displaystyle\sum_{j: j \neq i}a_{ij} & \text{if}~ i=j.
\end{cases} 
\end{equation}
Then the Monsky matrix is given by 
\begin{equation}\label{Monsky_matrix}
    M=\begin{bmatrix}
        A+D_2 & D_2 \\
        D_2 & A+D_{-2}
    \end{bmatrix}.
\end{equation}

\begin{lemma}\cite[Appendix]{Heath-Brown94}\label{Selmer rank}
    Let $M$ be the Monsky matrix defined as in \eqref{Monsky_matrix}. Then the $2$-Selmer rank $s_n=2r - \rank_{\mathbb{F}_{2}}(M).$
\end{lemma}

Let $n = p_{1} \cdots p_{r}$ be square-free, where $p_{1},\dots,p_{r}$ are distinct primes. Consider the set $G(n)$ of all positive divisors of $n$ i.e. $G(n) =\{a: ~a> 0, ~a\mid n\}.$ A straightforward calculation yields that $G(n)$ has a group structure under the following binary operation
\begin{equation}\label{thegroupG}
u*u' := \frac{uu'}{(\gcd(u,u'))^2}, \quad  u, u^\prime \in G(n).
\end{equation}
Every $u \in G(n)$ can be written uniquely in the form
$u = \prod_{i=1}^{r} p_{i}^{\varepsilon_{i}(u)},$
where each exponent $\varepsilon_{i}(u) \in \{0,1\}$. Define a map 
\begin{align}
\label{varphi-def} 
\varphi : G(n) \longrightarrow \F_2^{r}  ~\text{ by }~  \varphi(u)= \big(\varepsilon_{1}(u), \dots, \varepsilon_{r}(u)\big)^{\mathsf T},  
\end{align}
 where $\varepsilon_{i}(u)=1$ if $p_{i}\mid u$ and  $\varepsilon_{i}(u)=0$  otherwise. One checks easily that $\varphi$ is a group isomorphism. Moreover for $u,u' \in G(n)$, this induces the isomorphism
\begin{align}
\label{GXG}
\Psi : G(n) \times G(n) &\longrightarrow \F_2^{2r} \\ 
(u,u') &\longmapsto 
       \big(\varepsilon_{1}(u),\dots,\varepsilon_{r}(u),\varepsilon_{1}(u'),\dots,\varepsilon_{r}(u')\big)^{\mathsf T} \nonumber.
\end{align}

Consider the following system of equations naturally coming from the $2$-descent method on the congruent number elliptic curve $E_n.$ For further details, the reader is referred to \cite[Appendix]{Heath-Brown94}.
    \begin{equation}\label{equation}
    \begin{aligned}
    uu'x^2 + ny^2 &= uz^2,\\
    uu'x^2 - ny^2 &= u'w^2,
    \end{aligned}
    \qquad (u,u^\prime) \in G(n)\times G(n).
    \end{equation}
The algebraic rank and the $2$-Selmer rank of the elliptic curve $E_{n}$ are determined by the number of pairs $(u,u') \in G(n)\times G(n)$ for which the system of equations \eqref{equation} admits nontrivial global and local solutions, respectively. The following lemma, due to Heath-Brown \cite{Heath-Brown93}, gives an explicit count of the global (integral) solutions as well as the local solutions at all places.

\begin{lemma}\cite[Lemma 1]{Heath-Brown93}\label{local sol}
        Let $n$ be a square-free positive integer. Then there are exactly $2^{r_n}$ systems \eqref{equation} with non-trivial integer solutions. Moreover, there are $2^{s_n}$ systems in \eqref{equation} which are everywhere locally solvable. \qed
    \end{lemma}

 For each prime factor $p$ of $n$, the conditions for solvability of the system \eqref{equation} in $\Q_p$ are as follows:

 \begin{itemize}
    \item $\legendre{u}{p}=\legendre{u'}{p}=1$ for $p \nmid u$, $p\nmid u'.$
    \item $\legendre{2n/u}{p}=\legendre{2u'}{p}=1$ for $p \mid u,$ $p\nmid u'.$
    \item $\legendre{2u}{p}=\legendre{-2n/u'}{p}=1$ for $p \nmid u,$ $p\mid u'.$
    \item $\legendre{n/u}{p}=\legendre{-n/u'}{p}=1$ for $p \mid u,$ $p\mid u'.$
\end{itemize}

One can define a homomorphism $\phi_p : G(n) \times G(n) \to \{\pm 1\} \times \{\pm 1\}$ by
\begin{equation}\label{phi_p_def}
\phi_p(u,u') = 
\begin{cases}
\left(\legendre{u}{p},\,\legendre{u'}{p} \right), & p \nmid uu',\\[4pt]
\left(\legendre{2}{p},\,\legendre{-2}{p}\right), & (u,u')=(1,n),\\[4pt]
\left(\legendre{2}{p},\,\legendre{2}{p}\right), & (u,u')=(n,1).
\end{cases}
\end{equation}
Thus the system \eqref{equation} is solvable in $\mathbb{Q}_p$ precisely when $(u,u')$ lies in the kernel of $\phi_p$. The following lemma tells us about the cardinality of the kernel of $\phi_p$.
\begin{lemma}\label{2.4} \cite[Appendix]{Heath-Brown94}
 Let $n$ be a square-free integer and set $Z := \bigcap_{p \mid n} \ker(\phi_p),$ where $\phi_p$ is defined in \eqref{phi_p_def}. Then $(u,u') \in Z$ if and only if $\Psi(u,u')$ lies in the null space of the Monsky matrix $M$, where $M$ is given in \eqref{Monsky_matrix} and $\Psi$ is defined in \eqref{GXG}.  
Moreover, $\# Z = 2^{s_n}.$

\end{lemma}
We first establish Propositions \ref{snfor5557} and \ref{snfor553}, which play a crucial role in the proof of Theorems \ref{5557} and \ref{553}. Proposition \ref{snfor5557} determines the $2$-Selmer rank of the congruent number elliptic curve $E_n$ for $n$ as in Theorem \ref{5557}, while Proposition \ref{snfor553} provides an analogous result for $E_n$ with $n$ as in Theorem \ref{553}. We first recall the \textit{Guttman rank additivity formula}, relating the rank of a matrix to the ranks of its blocks.

\begin{proposition}\cite[Page 14]{zhang} \label{rank of bdm}
Let  { \scriptsize $C=\begin{bmatrix}
C_1 & C_2 \\
C_3 & C_4
\end{bmatrix}$}
be a block matrix over a field $\mathbb{F}$, where 
$C_1\in M_k(\mathbb{F})$, 
$C_2\in M_{k\times l}(\mathbb{F})$, 
$C_3\in M_{l\times k}(\mathbb{F})$, and 
$C_4\in M_l(\mathbb{F})$. 
If $C_1$ is invertible, then
{\small $\mathrm{rank}_{\mathbb{F}}(C)
= \mathrm{rank}_{\mathbb{F}}(C_1)
+
\mathrm{rank}_{\mathbb{F}}\!\left(C_4 - C_2 C_1^{-1} C_3\right).$}
\end{proposition}
The following lemma relates the rank of a symmetric matrix $A$ to the rank of its square $A^2$ over the field $\F_2$, and is used in the proofs of Propositions~\ref{snfor5557} and~\ref{snfor553}.

\begin{lemma}\label{rankofA}
    Let $A \in M_r(\mathbb{F}_2)$ be a symmetric matrix over $\mathbb{F}_2$ such that the sum of the entries in each row of $A$ is zero. Assume further that $\rank_{\F_2}(A)=r-1.$ Then
     $$\rank_{\F_2}(A^2)=
     \begin{cases}
         r-1, ~~\text{if} ~~ r ~~\text{is odd},\\
         r-2,  ~~\text{if} ~~ r ~~\text{is even}.
     \end{cases}$$
\end{lemma}

\begin{proof}
\begin{itemize}
    \item [Case 1:]Let $r$ be odd. 
    Since each row sum of $A$ is zero, we have $A\cdot\mathbf{1}_{r}=\mathbf{0}_{r}$.  Hence $\mathbf{1}_{r}\in \ker A$. Since nullity($A$)$=1$, $\ker A=\langle\mathbf{1}_{r}\rangle.$ To show $\rank_{\F_2}(A^2)=r-1,$ it is enough to show $\ker A^2=\langle\mathbf{1}_{r}\rangle$. Suppose nullity$(A^2)>1,$ then there exists $v \in \mathbb{F}_2^r,$ linearly independent of $\mathbf{1}_{r}$ such that $A^2v=\mathbf{0}_{r}$. Note that $Av\neq \mathbf{0}$. Now $A^2v=\mathbf{0}_{r}$ implies $Av \in \ker A$, hence $Av=\mathbf{1}_{r}.$ Taking the dot product on both sides with $\mathbf{1}_{r}$ yields $\mathbf{1}_{r}^{\mathsf T} A v = \mathbf{1}_{r}^{\mathsf T} \mathbf{1}_{r}= 1 $ (as $r$ is odd). However, $\mathbf{1}_{r}^{\mathsf T} A v = (Av)^{\mathsf T}A v= v^{\mathsf T}A^{\mathsf T}Av$, and since $A^{\mathsf T}=A$, $A^2v=\mathbf{0}_{r}$, it implies that $\mathbf{1}_{r}^{\mathsf T} A v= v^{\mathsf T}A^2v =0$. Hence we arrive at a contradiction. Therefore, $\ker A^2 = \langle \mathbf{1}_{r} \rangle.$

\item [Case 2:] Assume that $r$ is even. Note that $\ker A = \langle \mathbf{1}_{r} \rangle$ and we have $\ker A \subset \ker A^2$. Also, if $ v\in \ker A^2$, then $A v \in \langle \mathbf{1}_{r} \rangle$. Take any $u \in {\rm Im}(A)$. Then there exists $y \in \mathbb{F}_2^r$ such that $A y = u$. Because $A$ is symmetric, one checks that for all $u \in {\rm Im}(A)$ and all $x \in \ker A$, we have $x^{\mathsf T} u = 0$. Since $\mathbf{1}_{r} \in \ker A$, it follows that $\mathbf{1}_{r}^{\mathsf T} u = 0$ for all $u \in {\rm Im}(A)$, which means $\sum_{i=1}^r u_i = 0$. Thus ${\rm Im}(A)$ is contained in the set of all vectors $u \in \mathbb{F}_2^r$ with coordinate sum equal to $0$. A dimension argument shows that equality holds. Because $r$ is even, we have $\mathbf{1}_{r} \in {\rm Im}(A)$, so there exists some vector $v \neq \mathbf{1}_{r}$ with $A v = \mathbf{1}_{r}$. This shows that $\langle \mathbf{1}_{r}, v \rangle \subseteq \ker A^2$. Let $u \in \ker A^2 $ such that $\{u,v,\mathbf{1}_{r}\}$ is a linearly independent set. Then $A^2u=A^2v= \mathbf{0}_{r}$ which implies that both $Au, Av \in \ker A = \langle \mathbf{1}_{r} \rangle$. Hence $Au =  \mathbf{1}_{r} =Av$ which further implies that $u-v \in \ker A$. This gives $u-v= \mathbf{1}_{r}$ which contradicts the linearly independence of $u, v$ and $\mathbf{1}_{r}$. Hence $\ker A^2 = \langle \mathbf{1}_{r}, v \rangle$. Therefore $\operatorname{rank}_{\mathbb{F}_2}(A^2) = r - 2$.
    \end{itemize}
    \end{proof}

\begin{proposition}
\label{snfor5557}
Assume that $n$ is as in Theorem \ref{5557}. Then $s_n = 2$. Furthermore,
\[
Z = \{(1,1),\,(P,P),\,(q,1),\,(n,P)\},
\]
where $P := n/q = p_1 p_2 \dotsm p_t$, and $Z$ is as defined in Lemma \ref{2.4}.
\end{proposition}

\begin{proof}
Here $n = p_1 p_2 \dotsm p_t q$ with $t > 0$ odd and $(p_i, q) \equiv (5,7) \pmod{8}$ for all $1 \leq i \leq t$. 
Since $\legendre{2}{p_i} = \legendre{-2}{p_i} = -1$ for all $1 \leq i \leq t$, and $\legendre{2}{q} = 1$ while $\legendre{-2}{q} = -1$, we obtain $D_{2}=$ {\scriptsize $ \left[\begin{array}{c|c}
I_t & \mathbf{0}_t\\
\hline 
\noalign{\vskip 2pt} 
\mathbf{0}_t^{\mathsf T} & 0
\end{array}\right]$} and  $D_{-2} = I_{t+1}.$
Since $\legendre{q}{p_i} = 1$ for all $1 \leq i \leq t$, it follows that the $(t+1)\times(t+1)$ matrix $A$ defined in \eqref{diaij} is given by
$A= \left[\begin{array}{c|c}
A_P & \mathbf{0}_t\\[2pt]
\hline
\noalign{\vskip 2pt} 
\mathbf{0}_t^{\mathsf T} & 0
\end{array}\right]$. Hence, the $(2t+2)\times(2t+2)$ Monsky matrix defined in \eqref{Monsky_matrix} assumes the following form
 {\small $$M=\left[
\begin{array}{c|c|c|c}
A_P+I_t & \mathbf{0}_t & I_t & \mathbf{0}_t\\
\hline
\noalign{\vskip 2pt} 
\mathbf{0}_t^{\mathsf T} & 0 & \mathbf{0}_t^{\mathsf T} & 0 \\
\hline
\noalign{\vskip 2pt} 
I_t & \mathbf{0}_t & A_P+I_t & \mathbf{0}_t \\
\hline
\noalign{\vskip 2pt} 
\mathbf{0}_t^{\mathsf T} & 0 & \mathbf{0}_t^{\mathsf T} & 1
\end{array}
\right].
$$}
Observe that the $(t+1)$-th row (resp. $(t+1)$-th column) of $M$ is a zero row (resp. zero column). Moreover, noting that the $(2t+2)$-th row (resp. $(2t+2)$-th column) contains only a single nonzero entry, we obtain
\begin{equation}
\label{+1}
{\mathrm{rank}}_{\F_2}(M)= {\mathrm{rank}}_{\F_2}(M')+1
\end{equation}
where  
$M'=${\scriptsize 
$\left[
\begin{array}{c|c}
A_P+I_t & I_t \\
\hline
\noalign{\vskip 2pt} 
I_t & A_P+I_t
\end{array}
\right]$}
 $\in M_{2t}(\F_{2}).$
Let $R_i$ (resp. $C_i$) denote the $i$-th row (resp. $i$-th column) of the matrix $M'$. Replacing the columns $C_i$ with $C_i+C_{t+i}$ for all $1 \leq i \leq t$, and then replacing the rows $R_{t+i}$ with $R_{t+i}+R_i$ for all $1 \leq i \leq t$, we obtain $M''=$ {\scriptsize$\left[
\begin{array}{c|c}
A_P & I_t \\
\hline
\noalign{\vskip 2pt} 
O_t & A_P
\end{array}
\right]$} $\in M_{2t}(\F_{2})$. Since the rank of a matrix is invariant under these row and column operations, we have $\operatorname{rank}_{\mathbb{F}_2}(M') = \operatorname{rank}_{\mathbb{F}_2}(M'')=\operatorname{rank}_{\mathbb{F}_2}$  {\scriptsize $\left[
\begin{array}{c|c}
 I_t & A_P \\
\hline
\noalign{\vskip 2pt} 
 A_P & O_t 
\end{array}
\right]$}. Hence, by Proposition \ref{rank of bdm}, we have 
$\operatorname{rank}_{\mathbb{F}_2}(M') = \operatorname{rank}_{\mathbb{F}_2}(I_t) + \operatorname{rank}_{\mathbb{F}_2}(A_P^2)$. Since $A_p$ is symmetric, Lemma \ref{rankofA} implies 
$\operatorname{rank}_{\mathbb{F}_2}(M') = 2t-1$.

 Hence ${\mathrm{rank}}_{\F_2}(M)=2t$ by \eqref{+1}. Since the prime factorization of $n$ involves $t+1$ distinct primes, by Lemma \ref{Selmer rank} we have $s_n=2$. Note that $\{(1,1), (P,P),\,(q,1),\,(n,P)\} \subseteq \ker(\phi_p) \subseteq Z$ for each prime $p\mid n$. Hence we have  $Z=\{(1,1), (P,P),\,(q,1),\,(n,P)\}$ by using Lemma \ref{2.4}.
\end{proof}
Let $P = p_1 \cdots p_t$ as in Theorem \ref{553}, where $t$ is even and the distinct primes $p_i$ are congruent to $5$ modulo $8$. As noted in Lemma \ref{rankofA}, we have $\mathbf{1}_{t} \in \mathrm{Im}(A_P)$. Let $v = [v_1, \dotsc, v_t]^{\mathsf T}$ be such that $A_P v = \mathbf{1}_{t}$. Define
\begin{equation}
 \label{PvPv'}  
P_v := \prod_{p_i \mid P} p_i, \quad \text{with } p_k = 1 \text{ if } v_k = 0, \quad \text{  and }
P_v' := \frac{P}{P_v}. 
\end{equation}
We have the following:

\begin{proposition}
\label{snfor553}
Assume that $n$ is as in Theorem \ref{553}. Then $s_n = 2$. Furthermore, 
\[
Z = \{(1,1),\,(P,P),\,(P_v, P_v'),\,(P_v', P_v)\},
\]
where $Z$ is defined as in Lemma \ref{2.4}, and $P_v$ and $P_v'$ are as in \eqref{PvPv'}.
  
\end{proposition}

\begin{proof}
Observe that $n = p_{1}\cdots p_{t}q$, where each $p_i$ is a prime with $p_i \equiv 5 \pmod{8}$, the prime $q \equiv 3 \pmod{8}$, and $t$ is even. Note that in this case  $D_2= I_{t+1}$ and $D_{-2} = $ {\scriptsize$\left[\begin{array}{c|c}
I_t & \mathbf{0}_t\\
\hline \noalign{\vskip 2pt} 
\mathbf{0}_t^{\mathsf T} & 0
\end{array}\right]$}.
Hence, the Monsky matrix is given by
{\small \begin{equation*}
     M=\left[
\begin{array}{c|c|c|c}
A_P+I_t & \mathbf{0}_t & I_t & \mathbf{0}_t\\
\hline \noalign{\vskip 2pt} 
\mathbf{0}_t^{\mathsf T} & 1 & \mathbf{0}_t^{\mathsf T} & 1 \\
\hline \noalign{\vskip 2pt} 
I_t & \mathbf{0}_t & A_P+I_t & \mathbf{0}_t \\
\hline \noalign{\vskip 2pt} 
\mathbf{0}_t^{\mathsf T} & 1 & \mathbf{0}_t^{\mathsf T} & 0
\end{array} 
\right] \in M_{2t+2}(\F_{2}).
 \end{equation*}}
Note that ${\mathrm{rank}}_{\F_2}(M)= {\mathrm{rank}}_{\F_2}(M')+2$ where  $M'=$ {\scriptsize $\left[
\begin{array}{c|c}
A_P+I_t & I_t \\
\hline \noalign{\vskip 2pt} 
I_t & A_P+I_t
\end{array}
\right] $} $\in M_{2t}(\F_{2})$. Since $ A_P $ is symmetric, it follows from Lemma \ref{rankofA}, together with an argument analogous to that used in the proof of Proposition \ref{snfor5557}, that the $2$-Selmer rank satisfies $ s_n = 2 $.

It is easy to see that the vector $(\mathbf{1}_t^{\mathsf T},\,0,\,\mathbf{1}_t^{\mathsf T}\,,0)^{\mathsf T} \in \ker(M) \subset \mathbb{F}^{2t+2}_{2}$. Hence, by Lemma~\ref{2.4}, we obtain $(P,P) \in Z=\bigcap_{p \mid n}\ker(\phi_p)$. Since the vector $\mathbf{1}_t \in \mathrm{Im}(A_P)$, there exists $v$ such that $A_Pv=\mathbf{1}_t$. Then clearly $(v^{\mathsf T},\,0,\,(\mathbf{1}_t+v)^{\mathsf T},\,0)^{\mathsf T} \in \ker(M)\subset \mathbb{F}^{2t+2}_{2}$. Note that $(v^{\mathsf T},\,0,\,(\mathbf{1}_t+v)^{\mathsf T},\,0)^{\mathsf T}=\Psi(P_v,P_v')$, where $\Psi$ is as in~\eqref{GXG} and $P_v,P_v'$ are as in~\eqref{PvPv'}. Consequently, $Z=\langle (P,P),(P_v,P_v')\rangle$.   \end{proof}

  \begin{remark}
    Assume that $n$ is a congruent number, then the Mordell-Weil rank $r_n \geq 1$, hence we have $s_n \geq 1.$ By Lemma \ref{local sol}, the system of equations \eqref{equation} admit non-trivial and pairwise co-prime integral solutions $(x, y, z, w)$ for $n$ for at least one non-trivial pair $(u, u')$. Since the existence of a global (integral) solution to \eqref{equation} implies the existence of a local solution, we conclude that the system \eqref{equation} has a non-trivial solution for at least one non-trivial pair $(u,u^\prime) \in Z$ where $Z$ is defined in Lemma \ref{2.4}.
\end{remark}

\begin{proposition}\label{r_n=2}
Suppose that $n = p_1p_2\cdots p_tq$ is as in Theorem~\ref{553}, and assume that $n$ is a congruent number. If $\Sha(E_n/\mathbb{Q})[2]=0$, then $r_n=2$.
\end{proposition}

\begin{proof}
    Since $n$ is a congruent number, we have $r_n >0$. By Proposition \ref{snfor553}, we have $s_n=2$ and by the inequality (\ref{inequality-selmer}), we have $0 <
    r_n \leq 2.$ Since $\Sha(E_n/\Q)[2]=0$, by Cassel-Tate pairing \cite{cassels1962_arithmetic_iv}, $\dim_{\F_2}\Sha(E_n/\Q)[2]$ is even.  Using (\ref{relation of selmer sha}) we have $r_n\equiv s_n \pmod2$ and since $0 < r_n \leq 2$, we have $r_n=s_n=2$ .
\end{proof}




\section{\texorpdfstring{$4$}{lg}-rank and \texorpdfstring{$8$}{lg}-rank of Class Group}\label{4 rank and 8 rank}
In this section, we explain the method to computation of the $4$-rank and $8$-rank of the class group of the imaginary quadratic field $\mathbb{Q}(\sqrt{-m})$, where $m$ is a positive square-free integer of a specific type. Throughout, we assume that $m$ is odd.
 Then the discriminant of $\mathbb{Q}(\sqrt{-m})$ is given by 
$$D(-m) =
\begin{cases}
-4m, & \text{if } m \equiv 1\pmod{4}, \\
-m, & \text{if } m \equiv 3 \pmod{4}.
\end{cases}$$
We now recall the notion of the $2^k$-rank of a finite abelian group.
For a finite abelian group $C$, its $2^k$-rank is defined as
\begin{align*}
r_{2^k}(C) 
&= \dim_{\F_2} \left(2^{k-1}C / 2^kC\right) = \dim_{\F_2} \left(C[2^k] / C[2^{k-1}]\right),
\end{align*}
where $C[\ell]$ denotes the $\ell$-torsion subgroup of $C$.
In this paper, we are particularly interested in the ideal class group $C(-m)$. Accordingly, throughout the paper we write $r_{2^k}(-m)$
to denote the $2^k$-rank of $C(-m)$.
By Gauss genus theory \cite{cox2022primes}, the $2$-rank of its class group $C(-m)$ is $r_2(-m)=r-1$ where $r$ is the number of distinct prime divisors of $D(-m)$. Hence, we have $h(-m) \equiv 0 \pmod {2^{r-1}}.$




Before diving into the study of $2$-primary part of class group of certain imaginary quadratic field, it is worth introducing a tool that will be largely
used throughout the paper, the Hilbert symbols. For a rational prime $p$ and rational numbers $a$ and $b$, we denote the Hilbert symbol by $\legendre{a,b}{p}$. Define 
$\varepsilon(a) = \frac{a-1}{2}$  and  $\omega(a) = \frac{a^2 - 1}{8}$ for an odd positive integer $a$. For a prime $p$, write $a = p^{\alpha}u$ and $b = p^{\beta}v$, where $uv$ is coprime to $p$. Then
$$ \legendre{a,b}{p}
=
\begin{cases}
(-1)^{\alpha \beta \varepsilon(p)} \legendre{u}{p}^{\beta} \legendre{v}{p}^{\alpha}, & \text{if $p$ is odd;}\\[1em]
(-1)^{\alpha \omega(v) + \beta \omega(u) + \varepsilon(u)\varepsilon(v)}, & \text{if $p=2$.}
\end{cases}$$

We now introduce the generalized Rédei matrix, which serves as an important tool for studying the $2$-primary part of the class group of imaginary quadratic number fields. 
Let $\ell_1, \ell_2, \ldots, \ell_r$  be distinct primes dividing the discriminant $D(-m)$ of the imaginary quadratic field $\mathbb{Q}(\sqrt{-m})$. The generalized Rédei matrix $R(-m)$ is the $r \times r$ matrix over $\mathbb{F}_2$ given by
\begin{equation}\label{Redei-Matrix}
    R(-m) := \left( \phi\legendre{\ell_j, D(-m)}{\ell_i} \right)_{1 \leq i,j \leq r},
\end{equation}
where the map $\phi$ is defined as in \eqref{eq:phi-def}; see \cite{YueRedei} for further details.

Observe that the entries of the generalized Rédei matrix involve certain Hilbert symbols. The following lemma shows that these symbols can be expressed in terms of simpler Legendre symbols.
\begin{lemma}\label{Hilbert symbol-simplification}
   Let $\ell_1, \ell_2, \ldots, \ell_r$ be the odd prime divisors of the discriminant $D(-m)$.
   Then $$\left(\dfrac{\ell_j,D(-m)}{\ell_i}\right)=
   \begin{cases}
       \displaystyle \legendre{\ell_j}{\ell_i}, & i\neq j,\\
\displaystyle \prod_{j:\, j\neq i}\legendre{\ell_j}{\ell_i}, & i=j.
   \end{cases}$$
\end{lemma}

\begin{proof}
     The result follows directly from Lemma 2.1 of \cite{qin20052}.
\end{proof}

 The following result, due to Waterhouse \cite{waterhouse1973pieces}, is useful for computing the $4$-rank and $8$-rank of the class group of imaginary quadratic fields. It is stated here in a slightly modified form to suit our setting. Here, $r_4(-m)$ and $r_8(-m)$ denote the $4$-rank and $8$-rank of the class group of $\mathbb{Q}(\sqrt{-m})$, respectively.

\begin{proposition}\label{Waterhouse 4,8 rank}
   Let $m$ be a square-free integer. Suppose that the discriminant $D(-m)$ of the field $\Q(\sqrt{-m})$ has $r$ distinct prime factors $\ell_1, \ell_2, \dots, \ell_r$. Let $d$ and $d'$ run over the square-free positive divisors of $D(-m)$. Then the following hold:

\begin{enumerate}
    \item There are $2^{r_4(-m)+1}$ values of $d$ satisfying $\legendre{d,D(-m)}{\ell_i} = 1$ for all $1 \leq i \leq r$.
    \item For each $d$ in (1), choose a solution $(x_d, y_d, z_d)$ of the equation $x^2 - 4 d y^2 - D(-m) z^2 = 0$ in positive integers with no common factor. Then there are $2^{r_8(-m)+1}$ values for which there exists $d'$ such that 
    \[
        \legendre{y_d, D(-m)}{\ell_i} = \legendre{d', D(-m)}{\ell_i} \quad \text{for all } 1 \leq i \leq r.
    \]
\end{enumerate}
\end{proposition}

\begin{remark}\label{Waterhouse-remark}
{ \rm
\begin{enumerate}
    \item  Let $m_{0}$ be the positive square-free divisor of the discriminant $D(-m)$ such that every prime dividing $D(-m)$ also divides $m_{0}$. Observe that $\legendre{d, D(-m)}{\ell_i} = 1$ for a positive square-free divisor $d$ of $D(-m)$ if and only if $R(-m)\mathbf{v_d} = \mathbf{0}_{r}$, that is, $\mathbf{v_d}$ lies in the null space of the generalized Rédei matrix $R(-m)$, where $\mathbf{v_d}= \varphi(d)$ for $d \in G(m_{0})$  and $\varphi$ is defined as in \eqref{varphi-def}. Also the condition $\legendre{d, D(-m)}{\ell_i} = 1$ for all $i=1,\ldots, r$ is equivalent to the local solvability at $\ell_i$ of the quadratic equation $x^2-4dy^2-D(-m)z^2=0$, and hence, by the Legendre (Hasse–Minkowski) criterion, to the existence of a primitive integral solution. Equivalently, $\mathbf{v_d}$ is in the null space of $R(-m)$ if and only if the quadratic equation 
   \begin{equation}\label{quadratic equation for Legendre}
        x_d^2-4dy_d^2-D(-m)z_d^2=0
   \end{equation}
    admits a solution $(x_d, y_d,z_d) \in \Z^\times \times \Z^\times \times \Z^\times$. It then follows from Proposition~\ref{Waterhouse 4,8 rank} that $r_4(-m) = \text{Nullity}(R(-m)) - 1.$
    
    \item  Consider the imaginary quadratic field $\Q(\sqrt{-m})$, where $m=\ell_1 \ell_2\cdots \ell_r$ satisfies $m\equiv 3\pmod{8}$, and assume that $r_4(-m)=1$. Let $K(-m)=\{1,d_1,d_2,m\}$ with $d_1d_2=m$ and $\legendre{d,D(-m)}{l_i}=1$ for all $d\in K(-m)$. For each $d\in K(-m)$, let $(x_d,y_d,z_d)\in\Z^\times\times\Z^\times\times\Z^\times$ be a primitive solution of \eqref{quadratic equation for Legendre}, and consider the linear system \begin{equation}\label{8-rank system of equ} 
    R(-m)\mathbf{v}=\mathbf{C}_d, 
    \end{equation}
 where $\mathbf{C}_d=(c_1,\dots,c_r)^{\mathsf T}$ with $c_i=\phi\!\left(\legendre{y_d,D(-m)}{l_i}\right)$ and $\phi$ as in \eqref{eq:phi-def}. The system is always solvable for $d\in\{1,m\}$ and satisfies $y_{dd'}=y_d y_{d'}$ for $d,d'\in K(-m)$. Hence, by Proposition~\ref{Waterhouse 4,8 rank}, we have $r_8(-m)=1$ if and only if the above system is solvable for $d=d_1$ (equivalently, $d=d_2$); see also Remark~2.8 of \cite{das2026quantitative}.


\end{enumerate}
}   
\end{remark}


\subsection{\texorpdfstring{$4$}{lg}-Rank of the class group} 
 We first observe that if $n$ is given as in the Theorem \ref{5557}, then $r_2(-n)=t$. We now begin by studying the $4$-rank associated with the imaginary quadratic number field associated with Theorem \ref{5557}.  Suppose $n$ is defined as in Theorem \ref{5557}, then the discriminant $D(-n)=-n$. The positive divisors of $D(-n)$ forms a group $G(n)$ as describe in (\ref{thegroupG}). Moreover, the group $G(n)$ can be identified with the vector space $\mathbb{F}_2^{t+1}$  via the map $\varphi$ as defined in (\ref{varphi-def}). The next lemma determines the $4$-rank of the class group of $\mathbb{Q}(\sqrt{-n})$ for such $n$, as described in Theorem~\ref{5557}.

\begin{lemma}\label{4-rank 5557}
    Let $F=\Q(\sqrt{-n})$ where $n$ is as in Theorem \ref{5557}, then $r_4(-n)=1.$
\end{lemma}

\begin{proof}
    Observe that as $n$ is given as in Theorem \ref{5557}, discriminant of $\mathbb{Q}(\sqrt{-n})$ is $D(-n) = -n=-p_1p_2\cdots p_tq$.  To compute the $4$-rank, we will use part (1) of Proposition \ref{Waterhouse 4,8 rank} and Remark \ref{Waterhouse-remark}. The generalized Rédei matrix $R(-n)$ defined in \eqref{Redei-Matrix} is a $(t+1)\times (t+1)$ matrix of the form 
     $R(-n)=$ {\scriptsize $\left[\begin{array}{c|c}
     A_P & \mathbf{0}_t\\
    \hline \noalign{\vskip 2pt} 
   \mathbf{0}_t^{\mathsf T} & 0
   \end{array}\right]$} where the matrix $A_P$ is defined as in \eqref{matrix A_n 5557}.  Since $\mathrm{rank}(A_P) = t-1$, the matrix $R(-n)$ has nullity $2$.  Applying part (1) of Remark~\ref{Waterhouse-remark}, we conclude that $r_4(-n)=1$.   
\end{proof}

Now let $n$ and $P$ be as in Theorem \ref{553}, then $r_2(-n)=r_2(-P)=t$ that is $h(-n) \equiv h(-P) \equiv 0 \pmod {2^{t}}$. We now turn to the computation of the $4$-rank of the class group of $\mathbb{Q}(\sqrt{-n})$ and $\mathbb{Q}(\sqrt{-P})$.

\begin{lemma}\label{4-rank 553}
    Let  $n$ be given as in Theorem \ref{553}. Then $r_4(-n)=r_4(-P)=1$.
\end{lemma}

\begin{proof}
    Using similar argument as in Lemma \ref{4-rank 5557}, it follows that $r_4(-n)=1$ where $n$ is given as in Theorem \ref{553}.

    Our aim is to prove $r_4(-P)=1.$ Since $P=p_1p_2\cdots p_t \equiv 1 \pmod 4$,  $D(-P)=-4P$. To compute the $4$-rank of the class group of $\Q(\sqrt{-P})$, Remark~\ref{Waterhouse-remark} indicates that it suffices to determine the nullity of the generalized Rédei matrix $R(-P) \in M_{t+1}(\F_{2})$, given by {\scriptsize $\left[\begin{array}{c|c} 0 & \mathbf{0}_t^{\mathsf T} \\ \hline \noalign{\vskip 2pt} \mathbf{1}_t & A_P \end{array}\right]$}. Note that $\text{rank}_{\F_2}(R(-P))=\text{rank}_{\F_2}(R(-P)')$ where $R(-P)'=\left[\begin{array}{c|c}
    \mathbf{1}_t & A_P 
\end{array}\right]$. As noted in Proposition~\ref{snfor553}, the vector $\mathbf{1}_t$ lies in the column space of $A_P$. It follows that $\text{rank}_{\F_2}(R(-P)') = t-1 = \text{rank}_{\F_2}(R(-P))$, which shows that the nullity of $R(-P)$ over $\F_2$ is $2$. Therefore, by part~(1) of Remark~\ref{Waterhouse-remark}, we conclude that $r_4(-P)=1$.

\end{proof}



\subsection{\texorpdfstring{$8$}{lg}-Rank of the class group}
We begin by focusing on the $8$-rank of the class group of the imaginary quadratic field corresponding to $n$, where $n$ is as in Theorem~\ref{5557}.  
By Gauss's genus theory, we have $r_2(-n)=t$, and Lemma~\ref{4-rank 5557} gives $r_4(-n)=1$. Hence, 
$(\mathbb{Z}/2\mathbb{Z})^{t-1} \oplus \mathbb{Z}/4\mathbb{Z} \subseteq C(-n)$, so that $h(-n) \equiv 0 \pmod{2^{t+1}}$.  
Our goal is to determine when $h(-n) \equiv 0 \pmod{2^{t+2}}$, or equivalently, to describe when the $8$-rank of $C(-n)$ satisfies $r_8(-n)=1$.  
We will give a criterion for $r_8(-n)=1$ in terms of the quartic residue symbol. To this end, we first recall the definition of the quartic symbol.

Let $l$ be a prime of the form $4k+1$, and let $k$ be a quadratic residue modulo $l$. Then there exists an integer $r$ such that $k \equiv r^2 \pmod{l}$. The quartic symbol is denoted by $\legendre[4]{k}{l}$ and is defined by
$$\legendre[4]{k}{l} = \legendre{r}{l} = \pm 1 \equiv k^{\frac{l-1}{4}} \equiv r^{\frac{l-1}{2}} \pmod{l},$$
where the sign is $+1$ if and only if $k$ is a quartic residue modulo $l$, or equivalently, $r$ is a quadratic residue modulo $l$. Now let $L = \prod_{i=1}^r l_i$ be a square-free positive integer with $l_i \equiv 1 \pmod{4}$ for all $i$, and let $a \in \mathbb{Z}$ satisfy $(a,L)=1$ and $\legendre{a}{l_i}=1$ for all $i$. Then the quartic residue symbol modulo $L$ is defined multiplicatively by $\legendre[4]{a}{L} = \prod_{i=1}^r \legendre[4]{a}{l_i}.$


\begin{lemma}\label{8 rank condition}
    Consider the natural numbers $n$ satisfying the conditions as in the statement of Theorem \ref{5557}. Then we have $r_8(-n)=1$ if and only if $\legendre[4]{q}{P}=-1$ where $P= p_1 \cdots p_t$.
\end{lemma}

\begin{proof}
Suppose that $n$ is as in Theorem \ref{5557}. Observe that $\{\varphi(1), \varphi(n)\}\subset \mathrm{Ker}(R(-n))$, where $\varphi$ is defined as in \eqref{varphi-def}; these correspond to the trivial elements in the null space of $R(-n)$. Moreover, it is easy to verify that $\legendre{q, D(-n)}{p_i}=\legendre{q, D(-n)}{q}=1$, and $\legendre{P, D(-n)}{p_i}=\legendre{P, D(-n)}{q}=1$. Consequently, $\{\varphi(1), \varphi(q), \varphi(P), \varphi(n)\}$ lies in the null space of $R(-n)$. To compute the $8$-rank, we will use this non-trivial elements $q,P$ in the quadratic equation \eqref{quadratic equation for Legendre}. By Remark \ref{Waterhouse-remark}(2), it is enough to consider the element $q$. Hence for the choice $d=q$, consider the solution $(x_d,y_d,z_d)$ of the equation $x^2-4dy^2-D(-n)z^2=0$ defined as in \eqref{quadratic equation for Legendre}. In this case, the generalized Rédei matrix $R(-n)$ in \eqref{Redei-Matrix} is a $(t+1)\times(t+1)$ matrix of the form {\scriptsize $\left[\begin{array}{c|c} A_P & \mathbf{0}_t \\ \hline \noalign{\vskip 2pt} \mathbf{0}_t^{\mathsf T} & 0 \end{array}\right]$}, where $A_P$ is as in \eqref{matrix A_n 5557}. Since each row sum of the matrix $A_{P}$ is zero, it follows immediately that each row sum of the generalized Rédei matrix $R(-n)$ is also zero. If the system $R(-n)\mathbf{v}=\mathbf{C}_d$ in \eqref{8-rank system of equ} has a solution, then we must have $c_1+\cdots+c_t+c_{t+1}=0$, where $\mathbf{C}_d=(c_1,\ldots,c_t,c_{t+1})^{\mathsf T}$. Since the last row of $R(-n)$ is identically zero, the final equation forces $c_{t+1}=0$, and therefore $c_1+\cdots+c_t=0$.  By Remark \ref{Waterhouse-remark}(2) and Lemma \ref{Hilbert symbol-simplification}, we have the condition $c_1+\cdots+c_t=0$ is equivalent to $\legendre{y_d}{p_1\cdots p_t}=\legendre{y_d}{P}=1$. Note that we have $\legendre{P}{y_d}= \legendre{y_d}{P}=1$ since $P \equiv 1 \pmod 4$. For $d=q$, the quadratic equation \eqref{quadratic equation for Legendre} takes the form
$$x^2 - 4q y^2 + (p_1 \cdots p_t q) z^2 = 0,$$
which can be rewritten as
\begin{equation}\label{eq:quartic-form}
4y^2 = qx^2 + P z^2.
\end{equation}


Observe that both $x,z$ are odd integers. Considering \eqref{eq:quartic-form} modulo $P$, we obtain $4y^2\equiv qx^2 \pmod P.$ Taking the quartic residue symbol modulo $P$ then yields
$$\legendre[4]{q}{P}= \legendre[4]{4y^2x^2}{P}= \legendre{2}{P}\legendre{y}{P}\legendre{x}{P}.$$
Since $\legendre{2}{P}=-1$, we conclude that
$\legendre[4]{q}{P}=-\,\legendre{y}{P}\,\legendre{x}{P}.$
Again since $x$ is odd, considering \eqref{eq:quartic-form} modulo $x$, we get $4y^2\equiv Pz^2 \pmod x.$ Hence, $\legendre{P}{x}=\legendre{x}{P}=1$.
Hence $\legendre[4]{q}{P}=-\,\legendre{y}{P}$.
Since $\legendre{y}{P}=1$, we have $$r_8(-n)=1 \iff \legendre[4]{q}{P}=-1.$$
We are done.
\end{proof}

Now we aim to the study $8$-rank of the class group of quadratic imaginary number fields $\Q(\sqrt{-n})$ and $\Q(\sqrt{-P})$ arising from $n$ where $n$ and $P$ are defined as in Theorem \ref{553}. Similarly, in this case also by using Gauss genus theory we obtain $r_2(-n)=r_2(-P)=t$ and by Lemma \ref{4-rank 553} we obtain $r_4(-n)=r_4(-P)=1$. Consequently, $(\mathbb{Z}/2\mathbb{Z})^{t-1} \oplus \mathbb{Z}/4\mathbb{Z} \subseteq C(-n)$ (and $C(-P)$),  and therefore $h(-n)$ (and $h(-P)$) $\equiv 0 \pmod{2^{t+1}}$. Next two lemmas will give us the information about the $8$-rank of the class group of quadratic imaginary number fields $\Q(\sqrt{-n})$ and $\Q(\sqrt{-P})$. 
\begin{lemma}\label{8-rank of 553}
    Let $n=Pq$ be given as in Theorem \ref{553}. Then $r_8(-n)=1$ if and only if $\legendre[4]{q}{P}=1$.
\end{lemma}

\begin{proof}
    Note that since $P$ has even number of primes in its prime factorization, we have $\legendre{2}{P}=1$. The rest of the proof follows using similar argument as in Lemma \ref{8 rank condition}.
\end{proof}

In order to prove Theorem \ref{553}, we recall the following criteria given by Jung and Yue \cite[Theorem 3.3]{jung20118} for the $8$-rank of the class group of the quadratic imaginary number field $\Q(\sqrt{-P})$ to be $1$ in terms of quartic residue symbols between the prime factors of the discriminant and residue class modulo $16$.

\begin{lemma}\label{8 rank of n_q}
   Let $F=\Q(\sqrt{-P})$ where $P=l_1l_2\cdots l_t$ be a product of distinct primes such that $l_i \equiv 1 \pmod 4$ and assume that $r_4(-P)=1$. Then there exists $P_1,P_2$ such that $P=P_1P_2$ and 
    each of the quadratic Diophantine equations $$x^2 - D(-P) y^2 -4 P_1 z^2 = 0 \quad \text{and} \quad x^2 - D(-P) y^2 - 4P_2 z^2 = 0$$ admit nontrivial solutions over $\Z$. We have the following: 
   \begin{enumerate}
       \item Let $P_1 \equiv P_2 \equiv 5 \pmod 8$. Then $r_8(-P)=1$ if and only if either $P \equiv 9 \pmod{16}$ and $\legendre[4]{2P_1}{P_2}\legendre[4]{2P_2}{P_1}=-1$ or $P \equiv 1 \pmod{16}$ and $\legendre[4]{2P_1}{P_2}\legendre[4]{2P_2}{P_1}=1$.

       \item Let $P_1 \equiv P_2\equiv 1 \pmod 8$. Then $r_8(-P)=1$ if and only if either $P \equiv 1 \pmod{16}$ and $\legendre[4]{2P_1}{P_2}\legendre[4]{2P_2}{P_1}=1$ or $P \equiv 9 \pmod{16}$ and $\legendre[4]{2P_1}{P_2}\legendre[4]{2P_2}{P_1}=-1$.
   \end{enumerate}
\end{lemma}
 Recall from Lemma \ref{4-rank 553} that $r_4(-P)=1$, and the integers $P_v$ and $P_v'$ from (\ref{PvPv'}) such that $P_vP_v'=P$ where $P$ is given as in Theorem \ref{553}. By Remark \ref{Waterhouse-remark}(1), we see that the Diophantine equation $ x_d^2-4dy_d^2-D(-m)z_d^2=0$ has non-zero integer solutions for $d=P_v$ and $P_v'$. Hence, applying Lemma~\ref{8 rank of n_q} with $\{P_1,P_2\}=\{P_v,P_v'\}$, we obtain a criterion for $r_8(-P)=1$.

\section{Proof of main results}\label{proof of the theorem}

 We now proceed to the proofs of the main theorems of this article. The proofs are based on the preparatory results developed in Section~ \ref{Selmer group and Monsky} and \ref{4 rank and 8 rank}, in particular Lemmas \ref{local sol}, \ref{4-rank 553}, \ref{8 rank condition}, \ref{8 rank of n_q} and Propositions~\ref{snfor5557}, \ref{r_n=2}. Combining these intermediate results, we derive the proofs of Theorems \ref{5557} and \ref{553}.

\begin{proof}[Proof of Theorem \ref{5557}] \label{Proof of thm 1}

Let $n$ be as defined in the statement of Theorem~\ref{5557}. Suppose further that $n$ is a congruent number. Then the elliptic curve $E_n(\mathbb{Q})$ admits a point of infinite order, i.e.\ its algebraic rank $r_n$ is positive. By Lemma~\ref{local sol} and Proposition~\ref{snfor5557}, it follows that the system of equations~\eqref{equation} possesses a non-trivial integral solution for at least one pair $(u,u') \in \{(P,P),(q,1),(Pq,P)\}$. Moreover, by Lemma~\ref{8 rank condition}, we have
\[
h(-n) \equiv 0 \pmod{2^{t+2}}
\quad \text{if and only if} \quad \legendre[4]{q}{P}=-1.
\]
Therefore, to establish the main theorem, it suffices to show that whenever $n$ is congruent, for each $(u,u') \in \{(P,P),(q,1),(Pq,P)\}$, we have $\left(\frac{q}{P}\right)_4 = -1$.

\noindent\underline{Case $(u,u^\prime) = (P,P)$:} Substituting $(u,u^\prime) = (P,P)$ in \eqref{equation} and dividing both the equations by $P$, we get 
\begin{equation}\label{equation (P,P)}
    Px^2+qy^2=z^2,  \quad Px^2-qy^2=w^2.
\end{equation}
Reducing the first equation of \eqref{equation (P,P)} modulo $P$ we obtain,
$qy^2\equiv z^2 \pmod P.$
Hence, \begin{equation} \label{quartic(P,P)}
    \legendre[4]{q}{P}=\legendre[4]{y^2z^2}{P}=\legendre{y}{P}\legendre{z}{P}.
\end{equation}
It is easy to observe that $x,z$ and $w$ are odd integers. We now claim that $ y \equiv 2 \pmod{4}$. Taking difference of two equations of \eqref{equation (P,P)} we have
\begin{equation}\label{equation (P,P)diff}
    2qy^2=z^2-w^2= (z+w)(z-w).
\end{equation}
We claim that $y \equiv 2 \pmod{4}$. If $y$ were odd, then reducing equation \eqref{equation (P,P)diff} modulo $8$ and comparing both sides would lead to a contradiction. Hence, $y$ must be even. Now suppose that $y$ is divisible by $4$. Reducing the second equation of
\eqref{equation (P,P)} modulo $8$, we obtain $5 \equiv 1 \pmod{8}$, which is
impossible. Therefore, $y$ is divisible by $2$ but not by $4$. Consequently, we
may write $y = 2y'$, where $y'$ is odd. From (\ref{quartic(P,P)}), we have
\begin{equation}
\label{quartic(P,P)1}
\legendre[4]{q}{P}=\legendre{2y^\prime}{P} \legendre{z}{P}= \legendre{2}{P} \legendre{y^\prime}{P} \legendre{z}{P}= - \legendre{y^\prime}{P}\legendre{z}{P}.    
\end{equation}
$$$$
Now first equation of \eqref{equation (P,P)} can be written as $Px^2-4qy{^\prime}^2=w^2.$ Reducing modulo $y^\prime$, we have $Px^2\equiv w^2 \pmod {y^\prime}$ which implies $\legendre{P}{y^\prime}=\legendre{y^\prime}{P}=1$.  Hence by \eqref{quartic(P,P)1}, we have  $\legendre[4]{q}{P}= -\legendre{z}{P}.$
Now, adding the two equations in \eqref{equation (P,P)} and then reducing modulo $z$, we obtain $2P x^{2} \equiv w^{2} \pmod{z}.$
This implies that $\legendre{2P}{z}=1$, and hence $\legendre{2}{z} = \legendre{P}{z}.$

Our next goal is to calculate  $\legendre{2}{z}.$
It is easy to check that $\gcd(z+w, z-w)=2$. Now from equation \eqref{equation (P,P)diff} we have $8qy{^\prime}^2=(z+w)(z-w).$ Note that since $z, w$ and $y'$ are odd, then either $z+w \equiv 2 \pmod 4$ and $z-w \equiv 0 \pmod 4$, or $z+w \equiv 0 \pmod 4$ and $z-w \equiv 2 \pmod 4$.   
Assume that $z+w \equiv 2 \pmod{4}$ and $z-w \equiv 0 \pmod{4}$.
Then we may write
$$
z+w = 2 q^{\alpha} y_1^2 \quad \text{and} \quad
z-w = 4 q^{1-\alpha} y_2^2,
$$
where $\alpha \in \{0,1\}$, $\gcd(y_1,y_2)=1$, and we set $y' = y_1 y_2$. Adding these two equations, we obtain
$z = q^{\alpha} y_1^2 + 2 q^{1-\alpha} y_2^2.$  Noting that $y_1^2 \equiv y_2^2 \equiv 1 \pmod{8}
\quad \text{and} \quad
q \equiv 7 \pmod{8}$, we see that $z \equiv \pm 1 \pmod{8}.$ Hence, in this case, $\legendre{2}{z} = 1.$ A similar argument in the remaining case also gives
$\legendre{2}{z} = 1.$ Therefore,
$\legendre{P}{z} = \legendre{z}{P} = 1.$
Consequently, $$\legendre[4]{q}{P} = -\legendre{z}{P} = -1.$$

\noindent\underline{Case $(u,u^\prime) = (q,1)$:} Substituting $(u,u^\prime) = (q,1)$ in \eqref{equation} we get
$$qx^2+ Pqy^2=qz^2, \quad qx^2- Pqy^2=qw^2.$$
Clearing $q$, above equations can be rewritten as 
\begin{equation}\label{equation(q,1)}
    x^2+ Py^2=z^2,  \quad  x^2- Py^2=qw^2.
\end{equation}
Reducing the second equation  of \eqref{equation(q,1)} modulo $P$ we get $x^2\equiv w^2q \pmod P.$
Hence, $\legendre[4]{q}{P}=\legendre[4]{w^2x^2}{P}= \legendre{w}{P}\, \legendre{x}{P}.$
We now compute the Legendre symbols $\legendre{w}{p}$ and $\legendre{x}{P}$.  Observe that $z$ and $w$ are odd integers, while $x \equiv 2 \pmod{4}$. Hence we may write
$x = 2x',$ with $x'$  odd. From the first equation of \eqref{equation(q,1)}, we obtain $(2x')^{2} + Py^{2} = z^{2}.$
Reducing this equation modulo $x'$ gives $Py^{2} \equiv z^{2} \pmod{x'},$
and therefore $\legendre{P}{x'} = 1.$ Consequently,
\[
\legendre{x}{P}
= \legendre{2x'}{P}
= \legendre{2}{P}\,\legendre{x'}{P}
= -\,\legendre{P}{x'}
= -1.
\]
It follows that $\legendre[4]{q}{P} = -\,\legendre{w}{P}.$ Next, reducing the second equation of \eqref{equation(q,1)} modulo $w$, we obtain
$x^{2} \equiv Py^{2} \pmod{w},$
which implies $\legendre{w}{P} = 1.$
Hence, $\legendre[4]{q}{P} = -1.$
\vspace{0.2cm}

\noindent\underline{Case $(u,u^\prime) = (Pq,P)$:}
 Substituting $(u,u^\prime) = (Pq,P)$ in \eqref{equation} we get
$$P^2qx^2+ Pqy^2=Pqz^2, \quad P^2qx^2- Pqy^2=Pw^2.$$
Clearing the factor $Pq$ from the above equations, we may rewrite them as
\begin{equation}\label{equation(Pq,P)}
    Px^2+ y^2=z^2, \quad 
    Px^2- y^2=qw^2.
\end{equation}
Considering the second equation of \eqref{equation(Pq,P)} modulo $P$ we get $-y^2\equiv w^2q \pmod P.$
Hence, $$\legendre[4]{q}{P}=\legendre[4]{-w^2y^2}{P}=\legendre[4]{-1}{P}\legendre{w}{P}\legendre{y}{P}= -\legendre{w}{P}\legendre{y}{P}.$$
Observe that $y, z$ and $w$ are odd. Reducing the second equation of \eqref{equation(Pq,P)} modulo $w$ we get $Px^2\equiv y^2 \pmod w$ which implies $\legendre{P}{w}=1$ i.e $\legendre{w}{P}=1.$
Hence, $\legendre[4]{q}{P}=-\legendre{y}{P}.$
Again reducing the first equation of \eqref{equation(Pq,P)} modulo $y$, we get $Px^2\equiv z^2 \pmod y$ which implies $\legendre{P}{y}=1$ i.e. $\legendre{y}{P}=1.$
Hence, $\legendre[4]{q}{P}=-1.$
Hence the proof.
\end{proof}
Note that when $t = 1$, we consider $A_P = [0]_{1 \times 1}$ (see \eqref{matrix A_n 5557}). We therefore obtain the following result.

\begin{corollary}
    Let $n=pq$ be a congruent number with $p \equiv 5 \pmod 8$ and $q \equiv 7 \pmod 8$ and $\legendre{q}{p}=1$, then $h(-n)\equiv 0 \pmod 8$.
\end{corollary}

\begin{remark}
    This shows that for $t=1$, Theorem \ref{5557} agrees with the main theorem in \cite{das2025class}. Moreover it generalizes this result to the product of arbitrarily many odd number of primes $p_i \equiv 5 \pmod{8}$.
\end{remark}

\subsection{Proof of the Theorem \ref{553}}  \label{Proof of thm 2}

\begin{proof}[Proof of Theorem \ref{553}]
Here $n=Pq$ is as in Theorem~\ref{553}.  By Gauss genus theory and Lemma \ref{4-rank 553}, both $h(-n)$ and $h(-P)$ are divisible by $2^{t+1}$.  Assume further that $n$ is a congruent number. 
To show that
\[
h(-n) \equiv h(-P)+2^{t+1} \pmod{2^{t+2}},
\]
it suffices to prove that the $8$-ranks of the class groups of $\Q(\sqrt{-n})$ and $\Q(\sqrt{-P})$ cannot be simultaneously equal to $1$.

Since $n$ is a congruent number, the elliptic curve $E_n(\Q)$ has a point of infinite order, and hence the Mordell--Weil rank $r_n$ is positive. 
Moreover, by Proposition~\ref{r_n=2}, we have $r_n = 2$. 
By Lemma~\ref{local sol}, the system of equations \eqref{equation} admits a non-trivial integral solution $(x,y,z,w)$ for each pair
\[
(u,u') \in \{(P,P), (P_v,P_v'), (P_v',P_v)\}.
\]
Therefore, it suffices to show that $h(-n) \not\equiv h(-P) \pmod{2^{t+2}}$
holds for at least one pair
\[
(u,u') \in \{(P,P), (P_v,P_v'), (P_v',P_v)\}.
\]
We consider the pair $(P_v,P_v')$. Substituting $(u,u^\prime) = (P_v,P_v')$ in \eqref{equation}, we get 
    \begin{equation*}
    \begin{split}
        Px^2+ny^2= P_v z^2 , \quad 
        Px^2-ny^2=P_v'w^2.
    \end{split}
    \end{equation*}
Observe that $z$ is divisible by $P_v'$. Dividing both of the above equations by $P$ and, for simplicity, denoting $\frac{z}{P_v'}$ again by $z$, we obtain
    \begin{equation}\label{soe553a}
        x^2+qy^2= P_v' z^2, \quad 
        x^2-qy^2=P_v w^2
    \end{equation}
By adding and subtracting the two equations in \eqref{soe553a}, we obtain
\begin{equation}\label{addsoe553asoe553b}
    2x^2 = P_v' z^2 + P_v w^2,
\end{equation}
and
\begin{equation}\label{subsoe553asoe553b}
    2q y^2 = P_v' z^2 - P_v w^2,
\end{equation}
respectively. Observe that $x,z,w$ are odd and $y$ is even. Write $y=2^ry'$ with $y'$ being odd.   
Reducing \eqref{subsoe553asoe553b} modulo $y'$, we have $\legendre{P_v'}{y'}=\legendre{P_v}{y'}$ which implies that $\legendre{y'}{P}=1=\legendre{P}{y'}$. Therefore, $\legendre{y}{P}=\legendre{2}{P}^r\legendre{y'}{P}=1$.
  Reducing \eqref{soe553a} modulo $P_v'$ and the second equation of \eqref{soe553a} modulo $P_v$, we get 
  \begin{equation}
  \label{quarticvv'}
      \legendre[4]{q}{P_v'}= \legendre{x}{P_v'} \legendre{y}{P_v'} \quad \text{and} \quad \legendre[4]{-q}{P_v}= \legendre{x}{P_v} \legendre{y}{P_v}.
  \end{equation}
Using the multiplicativity of the quartic residue symbol and multiplying the equations in \eqref{quarticvv'}, together with $\legendre{y}{P}=1$, we get 
\begin{equation}\label{4 rank of q mod P}
    \legendre[4]{q}{P}= \legendre[4]{-1}{P_v}\legendre{x}{P}
\end{equation}    
Furthermore, multiplying the equation \eqref{addsoe553asoe553b} by $2$ and reducing it modulo $P_v$ and $P_v'$ simultaneously, we get
\begin{equation*}
    \legendre[4]{2P_v'}{P_v}= \legendre{2}{P_v}\legendre{x}{P_v}\legendre{z}{P_v} \quad \text{and} \quad \legendre[4]{2P_v}{P_v'}= \legendre{2}{P_v'}\legendre{x}{P_v'}\legendre{w}{P_v'}.
\end{equation*}
Further multiplying the above two equations and using the fact that $\legendre{2}{P}=1$, we get 
\begin{equation}
\label{quartic symbol of Pv and Pv'}
    \legendre[4]{2P_v'}{P_v} \legendre[4]{2P_v}{P_v'}= \legendre{x}{P}\legendre{z}{P_v} \legendre{w}{P_v'}.
\end{equation}
 It is easy to see that \eqref{addsoe553asoe553b} implies $\legendre{z}{P_v}=\legendre{2}{z}$ and $\legendre{w}{P_v'}=\legendre{2}{w}$. Hence from \eqref{quartic symbol of Pv and Pv'}, we get $\legendre[4]{2P_v'}{P_v} \legendre[4]{2P_v}{P_v'} = \legendre{x}{P} \legendre{2}{z} \legendre{2}{w}$. Note that $P=p_1p_2\cdots p_t \equiv 1 \pmod{8}$. To prove the theorem, we look at the residue class of $P$ modulo $16$ and establish the result for each case separately.
 
\begin{itemize}
    \item [(a)] Let $P\equiv 9 \pmod{16}$. Then we have either $P_v \equiv P_v' \equiv 5 \pmod{16}$, or $P_v\equiv P_v' \equiv 13 \pmod{16}$, or $P_v \equiv 1 \pmod{16}, P_v'\equiv 9 \pmod{16}$, or $P_v \equiv 9 \pmod{16}, P_v'\equiv 1 \pmod{16}$. 
    
    First assume that $P_v \equiv P_v' \equiv 5 \pmod{16}$. Since $x$ is odd, from \eqref{addsoe553asoe553b} it follows that $2\equiv P_vz^2+P_v'w^2 \pmod{16}$ which further implies that either $z^2 \equiv 1 \pmod{16} $ and $w^2 \equiv 9 \pmod{16}$ or $z^2 \equiv 9 \pmod{16} $ and $w^2 \equiv 1 \pmod{16}$. This implies that $z\equiv \pm 1 \pmod{8}$ and $w \equiv \pm 3 \pmod{8}$, or $z\equiv \pm 3 \pmod{8}$ and $w \equiv \pm 1 \pmod{8}$ which further gives $\legendre{2}{z}\legendre{2}{w}=-1$. This shows that $\legendre[4]{2P_v'}{P_v} \legendre[4]{2P_v}{P_v'}=-\legendre{x}{P}$. From (\ref{4 rank of q mod P}) and (\ref{quartic symbol of Pv and Pv'}), $\legendre[4]{2P_v'}{P_v} \legendre[4]{2P_v}{P_v'}=-1 \iff \legendre{x}{P}=1  \iff \legendre[4]{q}{P}=-1$. Hence by Lemma \ref{8-rank of 553} and \ref{8 rank of n_q}, we have $r_8(-P)=1$ if and only if $r_8(-n)=0$. \smallskip

    The cases where $P_v\equiv P_v' \equiv 13 \pmod{16}$, or $P_v \equiv 1 \pmod{16}$ and $ P_v'\equiv 9 \pmod{16}$, or $P_v \equiv 9 \pmod{16}$ and $P_v'\equiv 1 \pmod{16}$, follow exactly similar way as in the previous case and gives the desired equality. \smallskip 

    \item [(b)] Next Assume that $P \equiv 1 \pmod{16}$. 
    In this case, we have either $P_v \equiv P_v'\equiv 1 \pmod{16}$, or $P_v \equiv P_v'\equiv 9 \pmod{16}$, or $P_v \equiv 5 \pmod{16}$ and $P_v'\equiv 13 \pmod{16}$, or $P_v \equiv 13 \pmod{16}$ and $P_v'\equiv 5 \pmod{16}$. 
    
    Let $P_v \equiv P_v'\equiv 1 \pmod{16}$. Since $x$ is odd, in this case also we have $2\equiv P_vz^2+P_v'w^2 \pmod{16}$. Hence it follows that $z^2\equiv w^2 \equiv 1 \pmod{16}$ or $z^2\equiv w^2 \equiv 9 \pmod{16}$ which further implies that $z \equiv w \equiv \pm 1 \pmod{8}$ or $z \equiv w \equiv \pm 3 \pmod{8}$. Therefore $\legendre{2}{z}\legendre{2}{w}=1$ which implies $\legendre[4]{2P_v'}{P_v} \legendre[4]{2P_v}{P_v'}=\legendre{x}{P}$. Now by \eqref{4 rank of q mod P} and \eqref{quartic symbol of Pv and Pv'}, we have $\legendre[4]{2P_v'}{P_v} \legendre[4]{2P_v}{P_v'} = 1 \iff \legendre{x}{P}=1 \iff \legendre[4]{q}{P}=-1$. Hence Lemma \ref{8-rank of 553} and \ref{8 rank of n_q} implies that $r_8(-P)=1$ if and only if $r_8(-n)=0$, or in other words both $r_8(-P)$ and $r_8(-n)$ can not be simultaneously $1$. \smallskip

    The cases where $P_v \equiv P_v'\equiv 9 \pmod{16}$, or $P_v \equiv 5 \pmod{16}$ and $P_v'\equiv 13 \pmod{16}$, or $P_v \equiv 13 \pmod{16}$ and $P_v'\equiv 5 \pmod{16}$ follow similarly as above.
\end{itemize} 
We are done.
\end{proof}

\section{Quantitative Estimates associated with theorem \ref{5557}}\label{Quantitative estimate}

In this section, we establish lower bounds for the number of non-congruent numbers satisfying Theorem \ref{5557}. For real-valued functions $a(x)$ and $b(x)$ defined over $\mathbb{R}$, we write $a(x) \sim b(x) \text{ if }\displaystyle \lim_{x \to \infty}\frac{a(x)}{b(x)}=1.$ For a positive real number $x$ and integer $t$, define the following sets, where $A_P$ is as in \eqref{matrix A_n 5557}:

$$
\begin{aligned}
\mathcal{N}_t  
&:= \left\{ n = p_1 \cdots p_t q : p_1,\ldots,p_t,q \text{ are distinct primes},\ p_i \equiv 5 \pmod{8},\ q \equiv 7 \pmod{8} \right\}, \\
\mathcal{P}_t 
&: =\left\{P = p_1 \cdots p_t : p_1,\ldots,p_t \text{ are distinct prime numbers and } p_i \equiv 5 \pmod{8}\right\}\\
\mathcal{P}_t^{\mathrm{rk}} 
&:= \{ P \in \mathcal{P}_t  :\rank_{\mathbb{F}_{2}}(A_{P}) = t-1 \}, \\
\mathcal{N}_t^{\mathrm{rk}}(x)
&:= \left\{\, n = p_1 \cdots p_t q \le x \;:\;n \in \mathcal{N}_t,\; P = p_1 \cdots p_t \in \mathcal{P}_t^{\mathrm{rk}} \,\right\},\\
\pi_{k}(x) &:= \text{ Number of square-free integer }n \leq x~ \text{ such that $n$ has exactly $k$ distinct prime factors }.
\end{aligned}
$$

\begin{proposition}\label{Quantative bounds main Proposition}
    Let $\mathcal{N}_t^{(0)}(x) := \{ n: n \in \mathcal{N}_t^{\mathrm{rk}}(x) ,~ r_8(-n) = 0 \}$, where $\mathcal{N}_t^{\mathrm{rk}}(x)$ is defined as above. Then $$\#\mathcal{N}_t^{(0)}(x) \sim \frac{1}{2^{2t+3}}\cdot \pi_{t+1}(x) \cdot p(t-1),$$ where $p(t-1)$ denotes the probability that  $\rank_{\mathbb{F}_{2}}(A_{P})=t-1$ with $P \in \mathcal{P}_t  $.
\end{proposition}

\begin{proof}
    Let $\mathbb{P}(\cdot)$ denote the natural density of primes satisfying a given condition. By Dirichlet's theorem, primes are equidistributed among the reduced residue classes modulo $8$. In particular,
    \[\mathbb{P}\left(p \equiv 5 \pmod{8} \right)= \lim_{x \to \infty} \frac{\#\{\, p \le x : p \text{ prime},\ p \equiv 5 \!\!\pmod{8} \,\}}{\#\{\, p \le x : p \text{ prime} \,\}} =\frac{1}{\phi(8)}=\frac{1}{4}.\] 
 Similarly, $\mathbb{P}\left(q \equiv 7 \pmod 8\right)=\frac{1}{\phi(8)}=\frac{1}{4}.$ Thus, the density factor coming from fixed $t$ primes in residue class $5$ modulo $8$ and one prime in $7$ modulo $8$ is equal to $\frac{1}{\phi(8)^{t+1}}=\frac{1}{4^{t+1}}$. Hence, as $x \to \infty$, $$\# \mathcal{N}_t^{\mathrm{rk}}(x) \sim \dfrac{1}{2^{2t+2}}\cdot \pi_{t+1}(x)\cdot p(t-1).$$
For $n \in \mathcal{N}_t^{\mathrm{rk}}(x)$, we have
$\mathbb{P}\left(\legendre[4]{q}{p_1\cdots p_t}=1\right) = \frac{1}{2}.$ By Lemma \ref{8 rank condition}, $r_8(-n)=0$ if and only if $\legendre[4]{q}{P}=1$ where $n= Pq$ and $P=p_1 \cdots p_t$. So we have,
 $$\#\mathcal{N}_t^{(0)}(x) \sim \frac{1}{2^{2t+3}}\cdot \pi_{t+1}(x)\cdot p(t-1).$$
\end{proof}

The following lemma allows us to compute the term $p(t-1)$, as defined in Proposition \ref{Quantative bounds main Proposition} for any positive integer $t$.

\begin{lemma}
    Let $p(t-1)$ be as defined in Proposition \ref{Quantative bounds main Proposition}. Then $p(0)=1$, and for $t\geq 2$, we have $$p(t-1)=
\prod_{\substack{1 \le i \le t-1, \\ i \text{ odd}}} \left( 1 - \frac{1}{2^{i}} \right).$$
\end{lemma}

\begin{proof}
    For $t=1$, the matrix $A_P=[0]_{1 \times 1}$, and therefore $p(0)=1$. Suppose that $t \geq 2$. For $P \in \mathcal{P}_t $, we have $P=p_1 p_2 \cdots p_t,$ where $p_i \equiv 5 \pmod{8}$ are distinct primes. Let $A_P$ be the $t \times t$ symmetric matrix over $\mathbb{F}_2$ as defined in \eqref{matrix A_n 5557}. Since the sum of the rows of $A_P$ is zero, its rank is at most $t-1$. Consequently, by deleting a suitable row and the corresponding column from $A_P$, we obtain a $(t-1) \times (t-1)$ symmetric matrix $\widehat{A}_P$ such that $\rank_{\mathbb{F}_{2}}(A_{P}) = \rank_{\mathbb{F}_{2}}(\widehat{A}_{P}).$ Thus, for each $P \in \mathcal{P}_t $, the matrix $A_P$ may be regarded as a $(t-1) \times (t-1)$ symmetric matrix over $\mathbb{F}_2$ with the same rank as $A_P$. To obtain the desired probability, it suffices to compute the probability that a $(t-1)\times(t-1)$ symmetric matrix has full rank. For $t \geq 2$, this probability is known to be $\prod\limits_{\substack{1 \le i \le t-1, \\ i \text{ odd}}} \left( 1 - \frac{1}{2^{i}} \right)$ (see \cite{brent2010determinants}).
\end{proof}

The following lemma provides an asymptotic estimates for $\pi_k(x)$.

\begin{lemma}\cite[Theorem 437]{hardy1979introduction}
    Suppose that $\pi_k(x)$ is defined as above. As $x \to \infty$, we have, $\pi_1(x) \sim \frac{x}{\log x},$ and for all integers $k \geq 2$, $\pi_k(x) \sim \frac{x(\log\log x)^{k-1}}{(k-1)!\,\log x}.$
\end{lemma}

\begin{theorem}
    Suppose that $t$ is a positive integer and $x$ is a positive real number. Let $Q_{t,x}$ denote the set of non-congruent numbers $n$ of the form given in Theorem \ref{5557} with $n \le x$. Let $\mathcal{N}_t^{(0)}(x)$ be defined as in Lemma \ref{Quantative bounds main Proposition}. Then, for sufficiently large $x$, we have $$\#Q_{t,x} \geq \#\mathcal{N}_t^{(0)}(x)\sim \frac{1}{2^{2t+3}}\cdot \pi_{t+1}(x) \cdot p(t-1).$$
\end{theorem}

\begin{proof}
If $n$ is a congruent number satisfying the hypotheses of Theorem \ref{5557} then $h(-n) \equiv 0 \pmod {2^{t+2}}.$ Also $h(-n) \equiv 0 \pmod {2^{t+2}}$ if and only if $r_8(-n)=1$. Hence if $n\in Q_{t,x}$, then  $r_8(-n)=0$. By Lemma \ref{8 rank condition}, $r_8(-n)=0$ if and only if $\legendre[4]{q}{P}=1$ where $P= p_1 \cdots p_t$. Hence the result follows from Proposition \ref{Quantative bounds main Proposition}.
\end{proof}

\section{Computational Verification and Concluding Remarks} \label{Computation}

To illustrate our results, we identify several congruent numbers and apply Theorems \ref{5557} and \ref{553}. All computations are verified using SageMath \cite{sage}.

In Table \ref{Table 1 5557}, we list few congruent numbers satisfying the hypotheses of Theorem \ref{5557} for $t=3$, and as a result, we see that the congruence condition $h(-n) \equiv 0 \pmod {2^5}$ is satisfied. Next, we identify some non-congruent numbers by checking when the above congruence fails, i.e., when $h(-n) \not\equiv 0 \pmod {2^5}.$ These values are collected in Table \ref{Table 2 5557}.

\newcolumntype{C}[1]{>{\centering\arraybackslash}p{#1}}

\begin{longtable}{|C{2.5cm}|C{1.2cm}|C{3cm}|C{4.2cm}|C{1.2cm}|}
\caption{Congruent Numbers} \label{Table 1 5557}\\
\hline
$n$ &
$q$ &
$p_1 \cdot p_2 \cdot p_3$ &
$\left(\legendre{p_1}{p_2}, \legendre{p_2}{p_3} , \legendre{p_3}{p_1}, \legendre{q}{p_i} \right)$ &
$h(-n)$ \\
\hline
\endfirsthead

\hline
$n$ &
$q$ &
$p_1 \cdot p_2 \cdot p_3$ &
$\left(\legendre{p_1}{p_2}, \legendre{p_2}{p_3} , \legendre{p_3}{p_1}, \legendre{q}{p_i} \right)$ &
$h(-n)$ \\
\hline
\endhead

\hline
\endfoot

\hline
\endlastfoot
$3789955$ & $199$ & $5 \cdot 13 \cdot 293$ & $(-1,-1,-1, 1)$ & $224$ \\
\hline
$16330355$ & $71$ & $5 \cdot 157 \cdot 293$ & $(-1,-1,-1, 1)$ & $1568$ \\
\hline
$46381411$ & $167$ & $29 \cdot 61 \cdot 157$ & $(-1,-1,-1, 1)$ & $1344$ \\
\hline
$50095091$ & $47$ & $61 \cdot 101 \cdot 173$ & $(-1,-1,-1, 1)$ & $2848$ \\
\hline
$50850571$ &  $31$ & $101 \cdot 109 \cdot 149$ & $(-1,-1,-1, 1)$ & $1248$ \\
\hline
$54971195$ & $239$  &  $5 \cdot 157 \cdot 293$  &  $(-1,-1,-1, 1)$ & $2400$ \\
\hline
$79317011$ & $71$ & $37 \cdot 109 \cdot 277$ & $(-1,-1,-1, 1)$ & $2560$ \\
\hline
$289483979$ & $167$ & $61 \cdot 157 \cdot 181$ & $(-1,-1,-1, 1)$ & $4800$ \\
\hline
$2803700651$ & $191$ &  $197 \cdot 269 \cdot 277 $ & $(-1,-1,-1, 1)$ & $23808$ \\
\hline
\end{longtable}


\begin{longtable}{|C{2.5cm}|C{1.2cm}|C{3cm}|C{4.2cm}|C{1.2cm}|}
\caption{Non-Congruent Numbers} \label{Table 2 5557}\\
\hline
$n$ &
$q$ &
$p_1 \cdot p_2 \cdot p_3$ &
$\left(\legendre{p_1}{p_2}, \legendre{p_2}{p_3} , \legendre{p_3}{p_1}, \legendre{q}{p_i} \right)$ &
$h(-n)$ \\
\hline
\endfirsthead

\hline
$n$ &
$q$ &
$p_1 \cdot p_2 \cdot p_3$ &
$\left(\legendre{p_1}{p_2}, \legendre{p_2}{p_3} , \legendre{p_3}{p_1}, \legendre{q}{p_i} \right)$ &
$h(-n)$ \\
\hline
\endhead

\hline
\endfoot

\hline
\endlastfoot
 $2445755$  & $191$ & $5 \cdot 13 \cdot 197$  & $(-1,-1,-1, 1)$  & $368$  \\
\hline
 $3637595$  & $191$ & $5 \cdot 13 \cdot 293$  & $(-1,-1,-1, 1)$  & $560$ \\
\hline
$3638395$ & $71$  & $5 \cdot 37 \cdot 277$ & $(-1,-1,-1, 1)$ & $304$ \\
\hline
$7130155$ & $31$ & $5 \cdot  157 \cdot 293$ &  $(-1,-1,-1, 1)$ & $368$ \\
\hline
$7856795$ & $31$ & $5 \cdot 173 \cdot 293$ & $(-1,-1,-1, 1)$ & $752$ \\
\hline
$11201603$ & $7$ &  $53 \cdot 109 \cdot277$ & $(-1,-1,-1, 1)$ & $1072$ \\
\hline
$13788827$ & $263$ & $13 \cdot 37 \cdot 109$ & $(-1,-1,-1, 1)$  & $1568$ \\
\hline
$23856307$ & $47$ & $53 \cdot 61 \cdot 157$ & $(-1,-1,-1, 1)$ & $688$ \\
\hline
$26287523$ & $47$ & $53 \cdot 61 \cdot 173$ & $(-1,-1,-1, 1)$ & $1328$\\
\hline
$31491299$ & $7$ & $109\cdot 149\cdot 277$ & $(-1,-1,-1, 1)$ & $2448$\\
\hline
\end{longtable}


In Table \ref{Table 1 553}, we list some congruent numbers satisfying the hypotheses of Theorem \ref{553} for $t=2$ implying that the congruence condition $h(-n) \equiv h(-P) +8 \pmod {2^{4}}$ satisfied. Next, we identify some non-congruent numbers by checking when the above congruence fails, i.e., when $h(-n) \equiv h(-P) \pmod {2^{4}}.$ These values are collected in Table \ref{Table 2 553}. All computations are verified using SageMath \cite{sage}.


\begin{longtable}{|C{2.5cm}|C{1.2cm}|C{3cm}|C{4.2cm}|C{1.4cm}|C{1.2cm}|}
\caption{Congruent Numbers} \label{Table 1 553}\\
\hline
$n$ & $q$ & $p_1 \cdot p_2$ &$\left( \legendre{p_1}{p_2}, \legendre{p_1}{q}, \legendre{p_2}{q} \right)$ & $h(-P)$ & $h(-n)$\\
\hline
\endfirsthead

\hline
$n$ & $q$ & $p_1 \cdot p_2$ &$\left( \legendre{p_1}{p_2}, \legendre{p_1}{q}, \legendre{p_2}{q} \right)$ & $h(-P)$ & h(-n)\\
\hline
\endhead

\hline
\endfoot

\hline
\endlastfoot
$1443$ & $3$ & $13\cdot 37$ & $(-1, 1, 1)$ &  $16$ & $8$\\
\hline
 $2035$ & $11$  & $5 \cdot 37$  & $(-1, 1, 1)$ & $16$ & $8$\\
\hline
$8515$ & $131$ & $5 \cdot 13$ & $(-1, 1, 1)$ & $8$ & $16$ \\
\hline
$11635$ & $179$ & $5 \cdot 13$ & $(-1, 1, 1)$ & $8$ & $16$ \\
\hline
$13715$ & $211$ & $5 \cdot 13$ & $(-1, 1, 1)$ & $8$ & $32$ \\
\hline
$14915$ & $19$ & $5 \cdot 157$ & $(-1, 1, 1)$ & $16$ & $56$ \\
\hline
$15635$ & $59$ & $5 \cdot 53$ & $(-1, 1, 1)$ & $8$ & $32$ \\
\hline
$18715$ & $19$ & $5 \cdot 197$ & $(-1, 1, 1)$ & $24$ & $16$ \\
\hline
$26315$ & $19$ & $5 \cdot 277$ & $(-1, 1, 1)$ & $48$ & $40$ \\
\hline
$86435$ & $59$ & $5 \cdot 293 $ & $(-1, 1, 1)$ & $16$ & $88$ \\
\hline
$120235$ & $139$ & $5 \cdot 173$ & $(-1, 1, 1)$ & $16$ & $56$ \\
\hline

\end{longtable}


\begin{longtable}{|C{2.5cm}|C{1.2cm}|C{3cm}|C{4.2cm}|C{1.4cm}|C{1.2cm}|}
\caption{Non-Congruent Numbers} \label{Table 2 553}\\
\hline
$n$ & $q$ & $p_1 \cdot p_2$ &$\left( \legendre{p_1}{p_2}, \legendre{p_1}{q}, \legendre{p_2}{q} \right)$ & $h(-P)$ & $h(-n)$\\
\hline
\endfirsthead

\hline
$n$ & $q$ & $p_1 \cdot p_2$ &$\left( \legendre{p_1}{p_2}, \legendre{p_1}{q}, \legendre{p_2}{q} \right)$ & $h(-P)$ & $h(-n)$\\
\hline
\endhead

\hline
\endfoot

\hline
\endlastfoot
 $2915$ & $11$ & $5 \cdot 53$ & $(-1, 1, 1)$ & $8$ & $24$  \\
\hline
$4251$ & $3$ & $13 \cdot 109$ & $(-1, 1, 1)$ & $16$ & $16$\\
\hline
$8635$ & $11$ & $5\cdot 157$ & $(-1, 1, 1)$ & $16$ & $16$ \\
\hline
$9035$ & $139$ & $ 5 \cdot 13$ & $(-1, 1, 1)$ & $8$ & $40$ \\
\hline
$16315$ & $251$ & $5 \cdot 13$ & $(-1, 1, 1)$ & $8$ & $24$ \\
\hline
$25715$ & $139$ & $5 \cdot 37$ & $(-1, 1, 1)$ & $16$ & $48$ \\
\hline
$27235$ & $419$ & $ 5 \cdot 13$ & $(-1, 1, 1)$ & $8$ & $40$ \\
\hline
$31915$ & $491$ & $5 \cdot 13$ & $(-1,1,1)$ & $8$ & $24$ \\
\hline
$34715$ & $131$ & $5 \cdot 53$ & $(-1, 1, 1)$ &  $8$ & $56$ \\
\hline
$51467$ & $107$  & $13 \cdot 37$ & $(-1, 1, 1)$ & $16$ & $64$ \\
\hline
\end{longtable}

\textbf{Concluding Remarks:} This article presents two main theorems highlighting a close connection between the arithmetic of congruent number elliptic curves and certain imaginary quadratic fields arising from congruent numbers. In particular, we show how the congruent number property imposes explicit restrictions on the $2$-primary part of the class group. Theorem~\ref{5557} demonstrates that for $n=p_1\cdots p_tq=Pq$ satisfying specific congruence and quadratic residue conditions, the congruent number property forces $2$-divisibility of $h(-n)$ when $t$ is odd. In contrast, Theorem~\ref{553} reveals that when $q\equiv 3 \pmod 8$ and $t$ is even, the congruent number condition induces a nontrivial congruence relation between $h(-n)$ and $h(-P)$. We conclude by outlining natural questions and directions for future research.

\begin{enumerate}
    \item The rank condition $\rank_{\F_2}(A_P)=t-1$ plays a crucial role in in the proofs of both the theorems. It would be interesting to investigate whether this assumption can be weakened, or whether it admits a more conceptual interpretation.

    \item In Theorem~\ref{553}, the finiteness of $\Sha(E_n/\mathbb{Q})[2]$ is assumed. A natural problem is to determine whether this hypothesis can be removed.

    \item In this work, we restrict our attention to square-free integers $n$ having exactly one prime factor congruent to $3\pmod 4$. A natural next step is to extend our results to integers of the form $$n= p_1 \cdots p_t q_1 \cdots q_s,$$ where $p_i \equiv 5 \pmod 8$ and $q_j \equiv 3 \pmod 8$ or $7 \pmod 8$. 

    \item We able to establish lower bounds for the number of non-congruent numbers satisfying Theorem \ref{5557}. It would be interesting if one can provide similar types of lower bounds for Theorem \ref{553}.

    \item Finally, it would be worthwhile to explore whether analogous divisibility or congruence phenomena for class numbers occur for other families of congruent numbers, for different congruence classes of primes, or in the setting of real quadratic fields. 
\end{enumerate}

\subsection*{\it Acknowledgement} The first author acknowledges the support of the DST-INSPIRE Faculty Fellowship [DST/INSPIRE/04/2024/004189]. The second author gratefully acknowledges the first author for his hospitality, valuable research discussion during the visit to IIT Kanpur, where the final part of this work was carried out. The second author also thanks him for providing an excellent research environment. The third author acknowledges the support provided by the Institute Postdoctoral fellowship at IIT Madras.


\nocite{*}
\bibliographystyle{siam}
\bibliography{references}

\end{document}